\documentclass[11pt]{article}

\usepackage{amsfonts}
\usepackage{amssymb}
\usepackage{amsmath}
\usepackage{mathrsfs}
\usepackage[colorlinks, linkcolor=blue, anchorcolor=blue, citecolor=blue]{hyperref}
\usepackage[normalem]{ulem}
\usepackage{color}
\usepackage{leftidx}
\usepackage{verbatim}
\usepackage{enumerate}
\usepackage{amsthm}
\usepackage{geometry}
\usepackage{authblk}
\usepackage[toc,page]{appendix}
\usepackage[T1]{fontenc}

\newcommand{\E}{\mathbb{E}}

\newcommand{\R}{\mathbb{R}}
\newcommand{\N}{\mathbb{N}}

\newcommand{\X}{\mathfrak{X}}

\newcommand{\bx}{\mathbf{x}}
\newcommand{\by}{\mathbf{y}}

\newcommand{\bv}{\mathbf{v}}

\newcommand{\bz}{\mathbf{z}}

\newcommand{\dW}{\dot{W}}

\newcommand{\cM}{\mathcal{M}}

\newcommand{\1}{\mathbf{1}}

\newcommand{\bxi}{\pmb{\xi}}

\newcommand{\balpha}{\pmb{\alpha}}

\newcommand{\bbeta}{\pmb{\beta}}

\newcommand{\lgl}{\langle}
\newcommand{\rgl}{\rangle}

\newcommand{\cI}{\mathcal{I}}
\newcommand{\cF}{\mathcal{F}}

\newcommand{\hf}{\widehat{f}}
\newcommand{\hg}{\widehat{g}}
\newcommand{\kH}{\mathfrak{H}}

\newcommand{\bs}{\mathbf{s}}
\newcommand{\bt}{\mathbf{t}}
\newcommand{\br}{\mathbf{r}}
\newcommand{\dmd}{\diamond}

\newcommand{\dT}{\mathbb{T}}

\newcommand{\cH}{\mathcal{H}}

\newcommand{\cB}{\mathcal{B}}

\newcommand{\var}{\mathrm{Var}}

\newcommand{\kX}{\mathfrak{X}}

\newcommand{\cN}{\mathcal{N}}

\newcommand{\cA}{\mathcal{A}}

\newcommand{\bw}{\mathbf{w}}

\providecommand{\keywords}[1]
{
  \textbf{\textit{Keywords:\ }} #1
}

\newtheorem{theorem}{Theorem}[section]
\newtheorem{proposition}[theorem]{Proposition}
\newtheorem{lemma}[theorem]{Lemma}
\newtheorem{corollary}[theorem]{Corollary}

\newtheorem{remark}[theorem]{Remark}

\newtheorem{innerhyp}{Hypothesis}
\newenvironment{hyp}[1]
 {\renewcommand\theinnerhyp{#1}\innerhyp}
 {\endinnerhyp}

\begin{document}
\numberwithin{equation}{section}

\title{Quantitative central limit theorems for the parabolic Anderson model driven by colored  noises}

\author{David Nualart\footnote{D. Nualart: University of Kansas, USA; Email: nualart@ku.edu; supported by   NSF Grant DMS 2054735}, Panqiu Xia\footnote{P. Xia (corresponding author): University of Copenhagen, Denmark;  Email: px@math.ku.dk} \, and Guangqu Zheng\footnote{G. Zheng: The University of Edinburgh, UK; Email: zhengguangqu@gmail.com }}

\date{\small\today}

\maketitle
\vspace{-0.7cm}
\begin{abstract}

In this paper, we study the spatial averages of the solution  to the parabolic Anderson model driven by a   space-time Gaussian homogeneous noise that is colored in time and space.  We establish  quantitative central limit theorems (CLT) of this spatial statistics under some mild assumptions, by using the Malliavin-Stein approach.  The highlight of this paper is the obtention of rate of convergence in the colored-in-time  setting, where one can not use   It\^o's calculus due to  the lack of martingale structure. In particular, modulo highly technical computations, we apply a modified version of second-order Gaussian Poincar\'e inequality to overcome this lack of martingale structure and   our work improves the results by Nualart-Zheng (2020 \emph{Electron. J. Probab.}) and Nualart-Song-Zheng (2021 \emph{ALEA, Lat. Am. J. Probab. Math. Stat.}).

\end{abstract}


\normalsize

\tableofcontents

\section{Introduction}\label{sec_intro}

In the recent  years, the study of  spatial averages of the solution to certain stochastic partial differential equations (SPDEs) has received growing  attention. 
The paper  \cite{HNV20}, being the first of its kind,  investigated the  nonlinear stochastic heat equation on $ \R_+\times \R$
driven by a space-time white noise $\dot{W}$:
 Consider the stochastic heat equation on $\R_+ \times\R^d$ driven by a Gaussian noise $\dot{W}$:
\begin{align}  \label{SHE11}
 \begin{cases} 
 \frac {\partial u}{\partial t} = \frac{1}{2} \Delta u + \sigma(u) \dot{W} \\
  u(0,\bullet)  = 1 
  \end{cases},
\end{align}
where the nonlinearity is encoded into a deterministic Lipschitz continuous function $\sigma: \R\to\R$.
In Duhamel's formulation (mild formulation), this equation is equivalent to 
\begin{align} \label{int_form}
u(t,x)  =1 + \int_0^t \int_{\R} p_{t-s}(x-y) \sigma( u(s,y) ) W(ds,dy),
\end{align}
where the stochastic integral against $W(ds,dy)$ is an extension of the It\^o stochastic integral,
and
 $p_t(x) = (2\pi t)^{-1/2} e^{-  |x|^2 /(2t) }$ for $(t,x)\in\R_+\times\R$ denotes  the heat kernel.
    One of the main results in \cite{HNV20} provides the following estimate. Suppose 
also $\sigma(1)\neq 0$, which excludes the trivial case  $u(t,x) \equiv 1$. Let 
\[
F_R(t) =  \int_{-R}^R ( u(t,x) - 1 ) dx
\]
and let $d_{\rm TV}(X,Y)$ denote the total variation distance between two real random variables $X$ and $Y$ (see \eqref{def_tvd}). Then, it holds  that for any $t>0$, there is some constant $C_t$, that does not depend on $R$, such that the  following quantitative Central  Limit  Theorem (CLT) holds
  \begin{align} \label{SHE_rate}
 d_{\rm TV}\big(  F_R(t)/ \sigma_R(t)  , Z  \big) \leq  C_t  R^{-1/2},
 \end{align}
where $Z\sim \cN(0,1)$ is a standard normal  random variable and $\sigma_R(t) = \sqrt{ \text{Var}( F_R(t) )} > 0$ for each $t,R\in(0,\infty)$. The key ideas to obtain \eqref{SHE_rate} can be roughly summarized as follows:
\begin{itemize}
\item[(i)] By the mild formulation  \eqref{int_form} and applying stochastic Fubini's theorem, one can write
\[
F_R(t) = \int_{[0,t]\times\R} \Big( \int_{-R}^R p_{t-s}(x-y) dx \Big) \sigma( u(s,y) ) W(ds,dy) =: \delta( V_{t,R}  ),
\] 
where $\delta$ denotes the Skorohod integral (the adjoint of the Malliavin derivative operator; see Section \ref{sec_wcpam})  and  $V_{t,R}$ is the random kernel given by
 \[
V_{t,R}(s,y) =   \sigma ( u(s,y) ) \int_{-R}^R p_{t-s}(x-y) dx.
\]

\item[(ii)] Via standard computations, one can obtain $\sigma_R(t) \sim \text{constant}\times  R^{1/2}$ as $R\uparrow\infty$.

\item[(iii)]   The Malliavin-Stein bound (c.f. \cite[Proposition 2.2]{HNV20}), being the  most crucial ingredient, indicates that 
\begin{align}
 d_{\rm TV}\big(  \sigma_R(t)^{-1} F_R(t) , Z  \big) \leq  \frac{2}{\sigma_R(t)^2}   \sqrt{ {\rm Var} \big( \langle DF_R(t), V_{t,R} \rangle_{L^2(\R_+\times\R)}  \big)    }, \label{var_term}
\end{align}
where $DF_R(t)$ denotes the Malliavin derivative of $F_R(t)$, which is a random function and belongs to the space $L^2(\R_+\times\R)$ under the setting of \cite{HNV20}. Then the obtention of \eqref{SHE_rate} follows from a careful analysis of the inner product $\langle DF_R(t), V_{t,R} \rangle_{L^2(\R_+\times\R)} $.
\end{itemize}

We remark here that  for a general nonlinearity $\sigma$, the computations mentioned in points (ii) and  (iii) are made possible through applications of the \emph{Clark-Ocone formula} and \emph{Burkholder-Davis-Gundy inequality}, which are valid only in the white-in-time setting. The noise $\dot{W}$ that is white in time, naturally gives arise to a martingale structure so that It\^o calculus techniques come into the picture and enable the careful analysis of the variance term in \eqref{var_term}.

The above general strategy has also been exploited in several  other papers, see \cite{CKNP19-2,CKNP20,CKNP20-2,HNVZ19,KNP20,KY20} for results on stochastic heat equations and see \cite{BNZ20, DNZ20, NZ20} for results on stochastic wave equations, to name a few. The common feature of these papers is that they consider the case where the driving Gaussian noise is white in time so that the aforementioned strategy of \cite{HNV20} is working very  well. To the best of our knowledge, the colored-in-time setting has only been considered in  \cite{NSZ21,NZ19BM} for heat equations and in \cite{BNQSZ21} for  wave equations. 

\smallskip
In the present paper we are interested in the following parabolic Anderson model  (that means $\sigma(u)=u$)  on $\R_+\times\R^d$ driven by a Gaussian noise  $\dot{W}$, which is colored in time and space, with flat initial condition:
\begin{align}\label{pam}
\begin{cases}
\frac{\partial u}{\partial t}=\frac{1}{2}\Delta u+ u\dmd\dW,\\
u(0, \bullet)=1,
\end{cases}
\end{align}
where $\dmd$ denotes the Wick product (c.f. \cite[Section 6.6]{ws-16-hu}).  Because we allow the noise to be colored in time, we need to take $\sigma(u)=u$, otherwise it is not clear how to show the existence and uniqueness of a solution.

Let us now introduce some notation to better facilitate the discussion as well as  to state our main results.  Fix a positive integer $d$. Heuristically, $\dot{W}= \{ \dot{W}(t,x), (t,x) \in \R_+\times \R^d\}$ will be a centered Gaussian family of random variables with covariance structure given by
\[
\E[ \dot{W}(t,x)\dot{W}(s,y)] = \gamma_0(t-s) \gamma_1(x-y),
\]
where  $\gamma_0, \gamma_1$ are (generalized) functions that satisfy one of the following two conditions:
\begin{hyp}{1}\label{h1}
$(d\geq 1)$
$\gamma_0:\R\to[0,\infty]$ is a nonnegative-definite locally integrable function and $\gamma_1\geq 0$ is the Fourier transform of some nonnegative tempered measure $\mu$ on $\R^d$ {\rm(}called the spectral measure{\rm)}, satisfying Dalang's condition {\rm(}see {\rm \cite{Dalang99})},
\begin{align}\label{D_cond}
 \int_{\R^d} \frac{\mu(d\xi)}{1 + | \xi |^2} <+\infty.
\end{align}
\end{hyp}
\begin{hyp}{2}\label{h2}
 $(d=1)$
There are $H_0\in [\frac{1}{2},1)$ and $H_1\in (0,\frac{1}{2})$ with $H_0+H_1>\frac{3}{4}$, such that
\[
\gamma_0(t)=\begin{cases}
\delta(t),& H_0=\frac{1}{2},\\
|t|^{2H_0-2}, & \frac{1}{2} < H_0<1,
\end{cases}
\]
 where $\delta$ is the Dirac delta function at $0$ and $\gamma_1$ is the Fourier transform of $ \mu(d\xi) = c_{H_1}|\xi|^{1-2H_1}    d\xi$ with $c_{H_1} =  \pi^{-1}\int_{\R} (1- \cos x) |x|^{2H_1-2}dx$; see \eqref{c_a_22} for the choice of $c_{H_1}$.
 
 \end{hyp}
We will call the setting under the Hypothesis \ref{h1} the \textbf{regular case}, since the spatial correlation function $\gamma_1$ is a function as opposed to the setting under Hypothesis \ref{h2}. Oppositely, we call the setting under Hypothesis \ref{h2} the \textbf{rough case}, because the spatial correlation corresponds to fractional Brownian motion with Hurst index $H_1\in (0,1/2)$ (thus rougher than the standard Brownian motion or white noise).

In order to define rigorously the noise, we need some definitions.
Let $C_c^{\infty}(\R_+)$ and $C_c^{\infty}(\R^d)$ denote the set of real smooth functions with compact support  on $\R_+$ and $\R^d$, respectively.  Then, we define Hilbert spaces $\cH_0$ and $\cH_1$ to be the completion of $C_c^{\infty}(\R_+)$ and $C_c^{\infty}(\R^d)$ with respect to the inner products
\[
\lgl \phi_0, \psi_0\rgl_{\cH_0}  =\int_{\R_+^2}dsdt\gamma_0(t-s)\phi_0(t)\psi_0(s)
\]
and
\begin{align*}
\lgl \phi_1, \psi_1\rgl_{\cH_1}  &=\int_{\R^{2d}}dxdy\gamma_1(x-y)\phi_1(x)\psi_1(y)\\
&= \int_{\R^{d}} \mu(d\xi) \widehat{\phi}_1(\xi)\widehat{\psi}_1(-\xi),
\end{align*}
respectively, where $\widehat{\phi}_1(\xi)=\int_{\R^d}dxe^{-i\xi x}\phi_1(x)$  stands for the Fourier transform of $\phi_1$.

Set $\kH=\cH_0\otimes \cH_1$, equipped with the inner product
\begin{align}\label{def_1_real}
\lgl \phi, \psi\rgl_{\kH}  &= \int_{\R_+^2}dsdt \,\gamma_0(t-s)\int_{\R^{2d}}dxdy \gamma_1(x-y)\phi(t,x)\psi(s,y),
\end{align}
which can be also written using the Fourier transform as follows, 
\begin{align} \label{def_2_f}
\lgl \phi, \psi\rgl_{\kH}  &= \int_{\R_+^2}dsdt \,\gamma_0(t-s)\int_{\R^{d}} \mu(d\xi)  \widehat{ \phi}(t, \xi)\widehat{\psi}(s,-\xi),
\end{align}
where $\widehat{\phi}(t,\xi) $ and $\widehat{\psi}(t,\xi) $  stand for the Fourier transform in the space variable. 

We also introduce the following hypotheses that will be used to state our main results.
\begin{hyp}{3a}\label{h3a}
${\displaystyle \int_0^a \int_0^a \gamma_0(r-v)drdv > 0}$ for all $a>0$ and $0<\|\gamma_1\|_{L^1(\R^d)} < \infty$.
\end{hyp}
\begin{hyp}{3b}\label{h3b}
${\displaystyle \int_0^a \int_0^a \gamma_0(r-v)drdv > 0}$ for all $a>0$ and $\gamma_1(z) = |z|^{-\beta}$, $z\in\R^d$ for some $\beta\in (0, 2\wedge d)$.
\end{hyp}
We remark here that the restriction for $\beta$ in Hypothesis \ref{h3b} ensures Dalang's condition \eqref{D_cond}.

With these preliminaries, we consider a centered Gaussian family of random variables
 $W=\{W(h),h\in\kH\}$  with covariance structure 
\[
\E [W(\phi)W(\psi)]=\lgl \phi, \psi\rgl_{\kH}
\]
for all $\phi,\psi \in\kH$. 
The family $W$ is called an
  isonormal Gaussian  process over $\kH$.
 Heuristically, the noise $\dW (t,x_1, \dots ,x_d)=\frac{\partial^{d+1} W(t,x)}{\partial t\partial x_1 \cdots\partial x_d }$  is  the (formal) derivative of $W$ in time and space and the mild formulation of equation \eqref{pam} is  given by   
\begin{align}
u(t,x)  =1 + \int_0^t \int_{\R^d} p_{t-s}(x-y) \sigma( u(s,y) ) W(ds,dy),
\end{align}
where the stochastic integral against $W(ds,dy)$  is a Skorohod integral (see  \cite[Section 1.3.2]{Nualart06}).
  It has been proved  that under either Hypothesis \ref{h1} or \ref{h2}, the parabolic Anderson model \eqref{pam} admits a unique mild solution; see \cite{ ap-17-hu-huang-le-nualart-tindel,spde-17-huang-le-nualart, HHNT15,  SSX20}.

Due to the temporal correlation in time of the driving noise,  we do not have the playground to apply martingale techniques for obtaining central limit theorem for the spatial statistics
\begin{align}\label{def_FRT}
F_R(t) = \int_{B_R}  \big( u(t,x) - 1 \big) dx,
\end{align}
where $B_R=\{x\in \R^d,|x| \leq R\}$. Fortunately, because of the explicit chaos expansion (see \eqref{cos1} and \eqref{cos2}), one can express $  F_R(t)/\sigma_R(t)$, with $\sigma_R^2(t) = {\rm Var}\big(   F_R(t) \big)     $,  as a series of multiple stochastic  integrals.
This series falls into the framework of applying the so-called chaotic CLT. The chaotic CLT roughly means that once we have some control of the tail in the series, it would be enough to show the  convergence of each chaos, which can be further proved by using the fourth moment theorems; see \cite[Section 6.3]{blue} for more details. In fact, in the papers \cite{ NSZ21,NZ19BM}, the authors investigated the Gaussian fluctuations of $F_R(t)$ along this idea and proved the following results. 

\begin{theorem}\label{thm_quality}
 Assume Hypothesis \ref{h1} or \ref{h2}, and let $u$ be the solution to  the parabolic Anderson model \eqref{pam}. Recall the definition of $F_R(t)$ from \eqref{def_FRT} and let $\sigma_R(t) = \sqrt{  {\rm Var}(   F_R(t) )     }$.  Then the following results hold.

{\rm(1)} Assume Hypotheses \ref{h1} and \ref{h3a}. Then, for any fixed $t\in(0,\infty)$,  as $R\uparrow \infty$,
\begin{center}
$\sigma_R(t) \sim R^{d/2}$  and ${\displaystyle \frac{ F_R(t) }{\sigma_R(t) }   } $ converges in law to the standard normal law, 
\end{center}
where $a_R\sim b_R$ means $\displaystyle 0 < \liminf_{R\uparrow \infty} a_R/b_R\leq      \limsup_{R\uparrow \infty}a_R /b_R <+\infty$; see {\rm \cite[Theorem 1.6]{NZ19BM}}.

{\rm(2)}  Assume Hypotheses \ref{h1} and \ref{h3b}. Then, for any fixed $t\in(0,\infty)$,  as $R\uparrow \infty$,
\begin{center}
$\sigma_R(t) \sim R^{d- \frac{\beta}{2}}$  and ${\displaystyle \frac{ F_R(t) }{\sigma_R(t) }    } $ converges in law to the standard normal law;
\end{center}
 see {\rm \cite[Theorem 1.7]{NZ19BM}.}

{\rm(3)}  Under the hypothesis \ref{h2}, it holds for any fixed  $t\in(0,\infty)$ that  as $R\uparrow \infty$,
 \begin{center}
$\sigma_R(t) \sim R^{1/2}$  and ${\displaystyle \frac{ F_R(t) }{\sigma_R(t) }    } $ converges in law to the standard normal law;
\end{center}
see {\rm\cite[Theorem 1.1 and Proposition 1.2]{NSZ21}.}
\end{theorem}
\begin{remark}
More precisely, additionally to Theorem \ref{thm_quality} (3), we know from {\rm\cite[Equation (1.5)]{NSZ21}} that 
\begin{equation}
  \lim_{R\uparrow \infty}R^{-1} {\rm Var}(F_R(t)) = 2 \int_\R dz \E\big[ \mathfrak{g}( \mathcal{I}_z  )\big],   \label{eq1}
\end{equation}
where $\mathfrak{g}(z) = e^z -z - 1$ is strictly positive   except for $z=0$. Then the above limit vanishes if and only if almost surely   $\mathcal{I}_z =0$ for almost every $z\in\R$. Taking into account the explicit expression of $\mathcal{I}_z = \mathcal{I}_{t,t}^{1,2}(z)$ and  equation (1.8) in {\rm\cite[Proposition 1.3]{NSZ21}}, we can conclude that  the limit in \eqref{eq1} is strictly positive and thus $\sigma_R(t) \sim R^{1/2}$. Indeed, by (1.8) and $L^2$-continuity, $\E[ u(t,x) u(t,0) ] =  \E\big[ \exp( \mathcal{I}_x) \big]  >1$ for $x$ near $0$ so that with positive probability $\mathcal{I}_x \neq 0$ for $x$ near zero. 
\end{remark}

Note that the above CLT results are of qualitative nature and there are also functional version of these results where the limiting objects are  centered Gaussian process with explicit covariance structures. Both CLT in (1) and (3) are chaotic,  meaning that each\footnote{To be more precise, in case (3), the first chaotic component has negligible contribution  in the limit while  each of the other chaotic components has a Gaussian limit.} chaos contributes to the Gaussian limit, while CLT in (2) is not chaotic. More precisely, in (2) the first chaotic component, which is Gaussian, dominates the asymptotic behavior as $R\uparrow \infty$; see the above references for more details.  Here we point out that the application of the chaotic CLT does not yield the rate of convergence, that is, the  error bound like \eqref{SHE_rate} is not accessible through this method.  Our paper is devoted to deriving quantitative versions of the above CLT results as stated in the following theorem.

\begin{theorem}\label{thm_QCLT}
Let the assumptions in Theorem \ref{thm_quality} hold and recall cases {\rm(1) - (3)} therein.  Then we have for any fixed $t,R\in(0,\infty)$,
\begin{align}\label{main_TV}
d_{\rm TV}\big( F_R(t)/\sigma_R(t),Z \big) \leq   C_t \times
\begin{cases}
R^{-d/2} & \text{for \rm(1)} \\
R^{-\beta/2} & \text{for \rm (2)} \\
R^{-1/2} & \text{for \rm (3)} 
\end{cases},
\end{align}
where the constant $C_t>0$ is independent of $R$ and $Z\sim \mathcal{N}(0, 1)$ denotes a standard normal random variable.
\end{theorem}

In a  recent paper \cite{BNQSZ21}, the authors face the  problem of establishing a quantitative CLT for the hyperbolic Anderson model driven by a colored noise.
A basic ingredient in this paper   is the so-called second-order Gaussian Poincar\'e inequality (see Proposition \ref{propA1}).  With this inequality in mind, it is not difficult to see that, in order to obtain the desired rate of convergence, we need to equip ourselves with   fine  $L^p$-bounds of the Malliavin derivatives valued at (almost) \emph{every} space-time points.  This will be our approach in the \textbf{regular case},  and the majority of the effort will be allocated to show these $L^p$-bounds, with which we will  apply Proposition \ref{propA1} to get the quantitative CLT.

 However, in the  \textbf{rough case}  the spatial correlation $\gamma_1$  is a generalized function and  if one understands the inner product $\langle \bullet , \bullet \rangle_\kH $   using  the Fourier transform \eqref{def_2_f}, Proposition \ref{propA1} does not fit. This is a highly nontrivial difficulty the we overcome by taking advantage of an equivalent formulation of the inner product $\langle \bullet , \bullet \rangle_\kH $ based on fractional calculus (see Section \ref{sec_fss}).  
 Starting from such an equivalent expression, we derive in the \textbf{rough case} another version of the second-order Gaussian Poincar\'e inequality (see Proposition \ref{propA2}) that is better adapted to our purpose. We also refer the readers to Remark \ref{rmk_a} for a detailed discussion.

Let us complete this section with a few more remarks on (quantitative) CLT in   other settings.
\begin{itemize}
\item[1)]  The authors of  \cite{NSZ21}   study the situation where the noise $W$ is colored in space-time $\R_+\times\R$ with the spectral   density $\varphi$ satisfying a modified Dalang's condition and the concavity condition:
\[
\int_\R \frac{\varphi(x)^2}{1+ x^2}dx < +\infty\quad{\rm and}\quad \text{$\exists \kappa\in(0,\infty)$ such that  $\varphi(x+y) \leq \kappa( \varphi(x) + \varphi(y))$} 
\]
for all $x,y\in\R$. Using the chaotic CLT, they are able to establish the CLT results of qualitative nature. It is not clear to us how to derive the moment estimates for Malliavin derivatives in this setting.  A more intrinsic problem is that unlike in our \textbf{rough case}, we are not aware of any equivalent real-type expression for the inner product of the underlying Hilbert space $\kH$ and thus we do not see how to put the potential  moment estimates  in use.

\item[2)]    In the recent paper \cite{NZ20erg}, the spatial ergodicity for certain  nonlinear stochastic wave equations with spatial dimension not bigger than $3$ is established under some mild assumptions. The condition on the driving noise $W$ can be roughly summarized as follows: $W$ is white in time,  the spatial correlation satisfies  Dalang's condition and the spectral measure has no atom at  zero.  So far the CLT is open when the spatial dimension is three, while  there have been several works \cite{BNZ20, DNZ20, NZ20}  in dimensions $1$ or $2$.  When the spatial dimension is three, the wave kernel is a measure (not a function any more) and then the Malliavin derivative $Du(t,x)$ of the solution \emph{shall} be a random measure supported on some sphere, which can not be valued pointwisely.   
\end{itemize} 
 
 The rest of the paper is organized as follows: Section \ref{sec2} contains some preliminaries that will be used in this paper. We study the \textbf{regular case} in Section \ref{sec_regular} and leave the \textbf{rough case} to Section \ref{sec_rough}.

\section{Preliminaries and technical lemmas}\label{sec2}
In this section, we provide some preliminaries and useful lemmas.  

In the first step, we introduce some notation that will be used frequently in this paper. Let $n$ be a positive integer, we make use of notation $\bx_n=(x_1,\dots, x_n)\in \R^{nd}$ and $\bt_n=(t_1,\dots, t_n)\in \R_+^n$. Given $\bx_n\in \R^{nd}$, we write $\bx_{k:n}\in \R^{n-k+1}$ short for $(x_k,\dots, x_{n})$ with $k=1,\dots, n$, and $\bt_{k:n}\in \R_+^{n-k+1}$ is defined in the same way. Let $0\leq s<t<\infty$. We put $\dT_n^{s,t}=\{\bs_n\in\R_+^n: s<s_1<\dots< s_n<t \}$ and $\dT_n^{t}=\dT_n^{0,t}$. We use $c$, $c_1$, $c_2$ and $c_3$ for some positive constants which may vary from line to line. Finally, we write $A_1\lesssim A_2 $ if there exists a constant $c$ such that $A_1\le cA_2$.

\subsection{Wiener chaos and parabolic Anderson model}\label{sec_wcpam}
Let $\kH$ be a Hilbert space of (generalized) functions on $\R_+\times \R$, and let $W=\{W(h),h\in \kH\}$ be an isonormal Gaussian process over $\kH$. Denote by $\cF=\sigma(W)$ the smallest $\sigma$-algebra generated by $W$.
Then, any $\cF$-measurable and square integrable random variable $F$  can be unique expanded into a series of multiple It\^{o}-Wiener integrals (c.f. \cite[Theorem 1.1.2]{Nualart06}),
\begin{align}\label{def_cos}
F=\E(F)+\sum_{n=1}^{\infty}I_n(f_n),
\end{align}
where for $n=1,2,\dots$, $I_n(f_n)$ is the multiple It\^{o}-Wiener integral of $f_n$, which   is a  symmetric function on $(\R_+\times\R^{d})^n$, meaning
\begin{align}\label{def_hdn}
f_n\in \kH^{\odot n}=\big\{&f\in \kH^{\otimes n}, f(\bt_n,\bx_n)=f_n(t_{\sigma(1)},\dots, t_{\sigma(n)},x_{\sigma(1)},\dots,x_{\sigma(n)}),\nonumber\\
&  \text{for all permutation $\sigma$ on } \{1,\dots, n\}\big\}.
\end{align}
In the space-time white noise  case, $I_n(f_n)$ can be understood as the $n$-folder iterate It\^{o}-Walsh's integral (c.f. \cite{Walsh}),
\[
I_n(f_n)=n!\int_{0<t_1<\dots<t_n<\infty}\int_{\R^n}f_n(\bt_n,\bx_n)W(dt_1,dx_1)\cdots W(dt_n,dx_n).
\]
For any $n$, we denote by $\mathbf{H}_n$ the $n$-th Wiener chaos of $W$, that is the collection of random variables of the form $F=I_n(f_n)$, with $f_n\in \cH^{\odot n}$. In any fixed Wiener chaos $\mathbf{H}_n$, the following inequality of hypercontractivity (c.f. \cite[Corollary 2.8.14]{blue}) holds 
\begin{align}\label{hyper_ineq}
  \big\| F \big\|_p \leq (p-1)^{n/2} \big\| F \big\|_2,
 \end{align}
for $p\geq 2$ and for all $n=1,2,\dots$ and $F\in \mathbf{H}_n$.

Assume Hypothesis \ref{h1}. Set $\mathfrak{X}=\kH$  when  $\gamma_0=\delta$. Then, $\mathfrak{X}=L^2(\R_+ ; \cH_1)$, and we have the following lemmas that are very helpful in Section \ref{sec32_MD}.

\begin{lemma} \label{lem_redu}
For any nonnegative function $f\in\mathfrak{X}^{\otimes n}$ supported in $( [0,t]\times\R^d)^n$, we have 
\[
\big\| f\big\|^2_{\mathfrak{H}^{\otimes n}} \leq  \Gamma_t^n \big\| f\big\|^2_{\mathfrak{X}^{\otimes n}},
\]
where $\Gamma_t : = \int_{[-t,t]}\gamma_0(s)ds$.
\end{lemma}
\begin{lemma}[{\cite[Proposition 1.1.2]{Nualart06} and \cite[Section 2]{BNQSZ21}}]\label{lem_prod}
Let $m,n \geq 1$ be integers and let $f$ and $g$ in $\mathfrak{X}^{\odot n}$ and $\mathfrak{X}^{\odot m}$ respectively. Then,
\[
I_n^\X(f) I_m^\X(g) = \sum_{r=0}^{n \wedge m} r! \binom{m}{r} \binom{n}{r} I_{n+m - 2r}^\X\big( f\otimes_r g),
\]
where $I_n^\mathfrak{X}$ denotes the the multiple It\^{o}-Wiener integral with respect to an isonormal Gaussian process over $\kX$ and $f\otimes_r g$  denotes the $r$-th contraction between $f$ and $g$, namely, an element in $\mathfrak{X}^{n+m-2r}$ defined by 
\begin{align*}
\big(f\otimes_r g\big)( \bt_{n-r},  \bx_{n-r},  \bs_{m-r},  \by_{m-r}    ) =     \big\langle  f\big( \bt_{n-r},  \bx_{n-r},\bullet\big),   g\big( \bs_{m-r},  \by_{m-r}  , \bullet \big) \big\rangle_{L^2(\R_+^r; \cH_1^{\otimes r})}.
\end{align*}
In particular, if in addition $f, g$ have disjoint temporal support\footnote{This means $f = 0$ outside $(J\times\R^d)^n$ and    $g = 0$ outside $(J^c\times\R^d)^m$ for some $J\subset\R_+$. Note that for $f, g$ non-symmetric having disjoint temporal support, the equality \eqref{prod_eq} still holds true. }, then 
\begin{align} \label{prod_eq}
I_n^\X(f) I_m^\X(g) = I^\X_{n+m}(f\otimes g)
\end{align}
and $I_n^\X(f)$, $I_m^\X(g) $ are independent. 
\end{lemma}

In the rest of this subsection, we provide the definition for the solution to the parabolic Anderson model \eqref{pam}. Let $u=\{u(t,x): (t,x)\in \R_+\times \R^d\}$ be a $\sigma(W)$-measurable random field such that $\E (u(t,x)^2)<\infty$ for all $(t,x)\in \R_+\times \R^d$. Then, due to \eqref{def_cos}, we can write
\[
u(t,x)=\E (u(t,x))+\sum_{n=1}^{\infty}I_n( f_n(\bullet, t,x))=\E (u(t,x))+\sum_{n=1}^{\infty}I_n( f_{ t,x,n}(\bullet)),
\] 
for all $(t,x)\in \R_+\times \R^d$, where $f_n:(\R_+\times \R^{d})^{n+1}\to \R$ and for any $(t,x)\in \R_+\times \R^{d}$, $f_{t,x,n}=f_n(\bullet, t,x)\in \cH^{\odot n}$ (see \eqref{def_hdn}). Then, $u$ is said to be Skorohod integrable with respect to $W$, if the following series is convergent in $L^2(\Omega)$,
\[
\delta(u)=\int_0^{\infty}\int_{\R^d}u(t,x)W(dt,dx)=\int_0^{\infty}\int_{\R^d}f_0(t,x)W(dt,dx)+\sum_{n=1}^{\infty}I_{n+1}(\widetilde{f}_{n}),
\]
where $\widetilde{f}_n$ denotes the symmetrization of $f_n$ in $(\R_+\times \R^d)^{n+1}$. Additionally, $u$ is said to be a (mild) solution to the parabolic Anderson model \eqref{pam}, if for every $(t,x)\in \R_+\times \R$, as a random field with parameters $(s,y)$, $p_{t-s}(x-y)u(s,y)\1_{[0,t]}(s)$ is Skorohod integrable and the following equation holds almost surely,
\[
u(t,x)=1+\int_0^t\int_{\R^d}p_{t-s}(x-y)u(s,y)W(ds,dy).
\]
 It has been proved (c.f. \cite[Section 4.1]{ptrf-09-hu-nualart}) that $u$ is a solution to \eqref{pam}, if and only if it has the following Wiener chaos expansion
\begin{align}\label{cos1}
u(t,x)= 1+ \sum_{n=1}^{\infty} I_n( f_{t,x,n}),
\end{align}
where the integral kernels   $f_{t,x,n}$ are given by
\begin{align}\label{cos2}
f_{t,x,n}(\bs_n,\bx_n)=\frac{1}{n!}p_{t-s_{\sigma(n)}}(x-x_{\sigma(n)})\cdots p_{s_{\sigma(2)}-s_{\sigma(1)}}(x_{\sigma(2)}-x_{\sigma(1)}),
\end{align}
with $\sigma$  the permutation of $\{1,\dots,n\}$ such that $0<s_{\sigma(1)}<\dots<s_{\sigma(n)}<t$ and $p_t(x)$ being the heat kernel in $\R^d$. The chaos expansion \eqref{cos1} and the expression  \eqref{cos2} will be used frequently in this paper.

\subsection{Fractional Sobolev spaces and an embedding theorem}\label{sec_fss}
In this subsection, we give a basic introduction to  fractional Sobolev spaces. They are closely related to $\cH_1$ under Hypothesis \ref{h2}. We also provide an embedding theorem for $\cH_0$. These will be used in Section \ref{sec_rough}. For a more detailed account on applications  of this topic to SPDEs, we refer  the readers to papers \cite{ap-17-hu-huang-le-nualart-tindel, abel-18-hu-huang-le-nualart-tindel,arxiv-19-hu-wang} and the references therein.

Following the notation in \cite{bsm-12-di-nezza-palatucci-valdinoci}, given parameters $s\in (0,1)$ and $p\geq 1$, the fractional Sobolev space $W^{s,p}(\R)$ is the completion of $C_c^{\infty}(\R)$ with the norm
\[
\|\phi\|_{W^{s,p}}=\big(\|\phi\|_{L^p}^p+[\phi]_{W^{s,p}}^p\big)^{\frac{1}{p}},
\]
where  $[\bullet]_{W^{s,p}}$ denotes the Gagliardo (semi)norm
\begin{align}\label{def_gag}
[\phi]_{W^{s,p}}=\Big(\int_{\R^2}dxdy\frac{|\phi(x)-\phi(y)|^p}{|x-y|^{1+sp}}\Big)^{\frac{1}{p}}.
\end{align}
 In particular, if $p=2$, $W^{s,2}$ turns out to be a Hilbert space. By using the Fourier transformation, one can write (c.f. \cite[Proposition 3.4]{bsm-12-di-nezza-palatucci-valdinoci})
\begin{align*}
[\phi]_{W^{s,2}}^2=C(s)\int_{\R}d\xi|\xi|^{2s}|\widehat{\phi}(\xi)|^2,
\end{align*}
with $C(s)=\pi^{-1}\int_{\R}\frac{1-\cos(\zeta)}{|\zeta|^{1+2s}} d\zeta$. Here the constant $C(s)$ is slightly different from that in \cite{bsm-12-di-nezza-palatucci-valdinoci}, because we use another version of the Fourier transformation. Therefore, assuming Hypothesis \ref{h2}, we see immediately that
\begin{align}
\|\phi\|_{\cH_1}^2=c_{H_1}\int_{\R}d\xi|\xi|^{1-2H_1}|\widehat{\phi}(\xi)|^2=[\phi]_{W^{\frac{1}{2}-H_1,2}}^2. \label{c_a_22}
\end{align}
In this paper, we will use both representations of the norm in $\cH_1$ via the Fourier transformation and the Gagliardo formulation.

 In the next part of this subsection, we introduce an embedding property for the Hilbert space $\cH_0$. Assume Hypothesis \ref{h2} with $H_0\in(\frac{1}{2}, 1)$, then there exist a continuous embedding $L^{1/H_0}(\R_+) \hookrightarrow \cH_0$ (c.f. M\'{e}min et al. \cite[Theorem 1.1]{spl-01-memin-mishura-valkeila}), namely, there exists a constant $c_{H_0}$ depending only on $H_0$ such that
\begin{align}\label{embd}
|\lgl f, g\rgl_{\cH_0}|\leq c_{H_0} \|f\|_{L^{1/H_0}}\|g\|_{L^{1/H_0}}=c_{H_0}\Big(\int_{\R_+^2}dsdt |f(t)g(s)|^{\frac{1}{H_0}}\Big)^{H_0},
\end{align}
for all $f,g\in \cH_0$. Combining this fact and Cauchy-Schwarz's inequality on the Hilbert space $\cH_1$, we can show that for all $H_0\in[ \frac{1}{2}, 1)$,
\begin{align*}
\big|\lgl \phi, \psi\rgl_{\kH}\big|  &\leq  \int_{\R_+^2}dsdt\, \gamma_0(t-s) \big\vert \langle \phi(t,\bullet),  \psi(s, \bullet)\rangle_{\cH_1} \big\vert\\
& \leq  \int_{\R_+^2}dsdt\, \gamma_0(t-s) \big\|    \phi(t,\bullet) \big\|_{\cH_1} \times  \big\|  \psi(s, \bullet) \big\|_{\cH_1} \\
& \leq  c_{H_0}\|\phi\|_{L^{1/H_0}(\R_+;\cH_1)}\|\psi\|_{L^{1/H_0}(\R_+;\cH_1)}.
\end{align*}
By iteration, we can write
\begin{align}\label{emdn}
\big|\lgl \phi, \psi\rgl_{\kH^{\otimes n}}\big|\leq c_{H_0}^n\|\phi\|_{L^{1/H_0}(\R_+^n;\cH_1^{\otimes n})}\|\psi\|_{L^{1/H_0}(\R_+^n;\cH_1^{\otimes n})}
\end{align}
for all $\phi, \psi\in \kH^{\otimes n}$ and    $n=1,2,3,\dots$.

\subsection{Second-order Gaussian Poincar\'{e} inequalities}
In this subsection, we provide two versions of the second-order Gaussian Poincar\'{e} inequality (see \cite[Theorem 1.1]{jfa-09-nourdin-peccati-reniert} for the   first version). As stated in Section \ref{sec_intro}, they will be used in estimating the total variance distance in \textbf{regular} and \textbf{rough} cases respectively. 

Denote by $D$ the Malliavin differential operator (c.f. \cite[Section 1.2]{Nualart06}). Let $\mathbb{D}^{2,4}$ stand for the set of twice Malliavin differentiable  random variables $F$ with 
\begin{align*}
\|F\|_{2,4}^4=&\|F\|_4^4+\big\| \|DF\|_{\kH}\big\|_4^4 + \big\| \| D^2F \|_{\kH\otimes\kH} \big\|_4^4\\
=&\E[ F^4] + \E\big[ \|DF\|^4_{\kH} \big] + \E\big[ \| D^2F \|^4_{\kH\otimes\kH} \big] <\infty.
\end{align*}
We denote by $\mathbb{D}_*^{2,4}$  the set of  random variables $F\in \mathbb{D}^{2,4}$ such that we can find versions of the derivatives $DF$  and $D^2F$, which are   measurable functions on $\R_+\times \R^d$ and $(\R_+\times \R^d)^2$, respectively  such that  $|DF| \in  \kH$ and $|D^2F| \in \kH^{\otimes 2}$.

Let $F$ and $G$ be random variables. The total variance distance between $F$ and $G$ is defined by
\begin{align}\label{def_tvd}
d_{\rm TV}(F,G) = \sup\{ | \mu(A) - \nu(A) |, A\subset \R\ \mathrm{is\ Borel\ measurable}\},
\end{align}
where $\mu, \nu$ are  the probability laws of $F$ and $G$ respectively. 
 
 
 
 

The next proposition cited from \cite[Proposition 1.8]{BNQSZ21} will be used in the \textbf{regular case}.
\begin{proposition}\label{propA1} 
Assume  Hypothesis \ref{h1}. Let $F\in\mathbb{D}^{2,4}_*$  be a random variable with mean zero and standard deviation $\sigma\in(0,\infty)$. Then
\[
d_{\rm TV} \big( F/\sigma, Z \big) \leq \frac{4}{\sigma^2} \sqrt{\mathcal{A} },
\]
where $Z\sim\mathcal{N}(0,1)$ and 
\begin{align*}
\mathcal{A} &:= \int_{\R_+^6 \times\R^{6d}} drdr' dsds' d\theta d\theta' dzdz' dwdw' dydy' \gamma_0(\theta - \theta') \gamma_0(s-s') \gamma_0(r-r') \\
 &\qquad \times \gamma_1(z-z') \gamma_1(w-w') \gamma_1(y-y') \big\| D_{r,z}D_{\theta,w} F \big\|_4  \big\| D_{s,y}D_{\theta',w'} F \big\|_4  \big\| D_{r',z' } F \big\|_4\big\| D_{s',y' } F \big\|_4.
\end{align*}
\end{proposition}
Inspired by \cite[Theorem 2.1]{jtp-20-vidotto}, we also have the following proposition, that will be used in the \textbf{rough case}.
 \begin{proposition} \label{propA2}
Assume Hypothesis \ref{h2}.  Let $F\in\mathbb{D}^{2,4}_*$ be a random variable with  mean zero and standard deviation $\sigma\in(0,\infty)$. Then
\[
d_{\rm TV} \big( F/\sigma, Z \big) \leq \frac{2\sqrt{3}}{\sigma^2} \sqrt{\mathcal{A} },
\]
where $Z\sim\mathcal{N}(0,1)$ and  
\begin{align}\label{def_ca2}
\mathcal{A} &:= \int_{\R_+^6} dsds'd\theta d\theta' drdr' \gamma_0(s-s') \gamma_0(\theta -\theta')\gamma_0(r-r') \int_{\R^6} dydy'dzdz'dwdw' |y-y'|^{2H_1-2} \nonumber\\
&\qquad\times  |z-z'|^{2H_1-2} |w-w'|^{2H_1-2} \| D_{r',z}F -  D_{r',z'} F\|_4 \| D_{\theta',w}F -  D_{\theta',w'} F\|_4 \nonumber\\
&\qquad\times  \big\| D_{r,z}D_{s,y}F - D_{r,z}D_{s,y'}F - D_{r,z'}D_{s,y}F + D_{r,z'}D_{s,y'}F  \big\|_4 \nonumber\\
&\qquad\times  \big\| D_{\theta,w} D_{s', y}F - D_{\theta,w} D_{s', y'}F - D_{\theta,w'} D_{s', y}F + D_{\theta,w'} D_{s', y'}F \big\|_4.
\end{align}
\end{proposition}
\begin{proof}
We begin with the Malliavin-Stein bound\footnote{see \cite[Theorem 5.1.3]{blue} and also equation (1.26) and footnote 4 in \cite{BNQSZ21}, where the latter reference points out that we do not need to assume the existence of density for $F$.}
\[
d_{\rm TV}\big( F/\sigma, Z \big) \leq \frac{2}{\sigma^2} \sqrt{  {\rm Var}\big(    \langle DF, - DL^{-1}F \rangle_{\kH} \big)     }
\]
with $L^{-1}$ denoting the pseudo-inverse of the Ornstein-Uhlenbeck operator. Denote by $P_t$ the Ornstein-Uhlenbeck  semigroup. Then, following the arguments verbatim  in \cite[Appendix 2]{BNQSZ21}, we  have 
\begin{align}
 {\rm Var}\big(    \langle DF, - DL^{-1}F \rangle_{\kH} \big)   \leq & 2 \E \int_0^\infty dt e^{-t} \big\langle D^2F\otimes_1 D^2F, P_t(DF) \otimes P_t(DF) \big\rangle_{\kH^{\otimes 2}}  \label{1st_term_in}\\
 & \quad+ 2  \E \int_0^\infty dt e^{-2t} \big\langle P_t(D^2F)\otimes_1 P_t(D^2F),  DF \otimes DF \big\rangle_{\kH^{\otimes 2}}, \notag
 \end{align}
where the two terms on the right-hand side can be dealt with in the same manner. In what follows, we only estimate the second term. Put 
\begin{align*}
  g(r,z,\theta, w) :&= P_t(D_{r,z} D_{\bullet}F)\otimes_1 P_t(D_{\theta, w} D_{\bullet}F) \\
& = \int_{\R_+^2}dsds' \gamma_0(s-s') \int_{\R^2 }dydy' |y - y'|^{2H_1-2}  \\
&\qquad \times P_t\big(D_{r,z} D_{s,y}F -D_{r,z} D_{s,y'}F\big) P_t\big(D_{\theta,w} D_{s',y}F -D_{\theta,w} D_{s',y'}F\big), 
\end{align*}
where we have used the expression  \eqref{def_gag}  for the inner product in $\kH_1$.
With this notation  we can write
\begin{align*}
&\quad \big\langle P_t(D^2F)\otimes_1 P_t(D^2F),  DF \otimes DF \big\rangle_{\kH^{\otimes 2}} = (g\otimes_1 DF) \otimes_1 (DF)     \\
&= \int_{\R_+^4} d\theta d\theta' drdr' \gamma_0(\theta -\theta')\gamma_0(r-r') \int_{\R^4} dzdz'dwdw' |z-z'|^{2H_1-2} |w-w'|^{2H_1-2} ( D_{r',z}F -  D_{r',z'} F)  \\
&\quad\times( D_{\theta',w}F -  D_{\theta',w'} F)\big[  g(r,z,\theta,w)-g(r,z', \theta,w)-g(r,z,\theta,w')+g(r,z',\theta,w')  \big] \\
&=  \int_{\R_+^6} dsds'd\theta d\theta' drdr' \gamma_0(s-s') \gamma_0(\theta -\theta')\gamma_0(r-r') \int_{\R^6} dydy'dzdz'dwdw' |y-y'|^{2H_1-2} \\
&\qquad\times  |z-z'|^{2H_1-2} |w-w'|^{2H_1-2}( D_{r',z}F -  D_{r',z'} F)( D_{\theta',w}F -  D_{\theta',w'} F) \\
&\qquad\times P_t\big( D_{r,z}D_{s,y}F - D_{r,z}D_{s,y'}F - D_{r,z'}D_{s,y}F + D_{r,z'}D_{s,y'}F  \big) \\
&\qquad\times P_t\big( D_{\theta,w} D_{s', y}F - D_{\theta,w} D_{s', y'}F - D_{\theta,w'} D_{s', y}F + D_{\theta,w'} D_{s', y'}F \big).
\end{align*}
 Therefore, by using H\"older inequality and the contraction property of $P_t$ on $L^4(\Omega)$, we get 
 \begin{align*}
& \quad 2  \E \int_0^\infty dt e^{-2t} \big\langle P_t(D^2F)\otimes_1 P_t(D^2F),  DF \otimes DF \big\rangle_{\kH^{\otimes 2}} \\
 &\leq  \int_{\R_+^6} dsds'd\theta d\theta' drdr' \gamma_0(s-s') \gamma_0(\theta -\theta')\gamma_0(r-r') \int_{\R^6} dydy'dzdz'dwdw' |y-y'|^{2H_1-2} \\
&\qquad\times  |z-z'|^{2H_1-2} |w-w'|^{2H_1-2} \| D_{r',z}F -  D_{r',z'} F\|_4 \| D_{\theta',w}F -  D_{\theta',w'} F\|_4 \\
&\qquad\times  \big\| D_{r,z}D_{s,y}F - D_{r,z}D_{s,y'}F - D_{r,z'}D_{s,y}F + D_{r,z'}D_{s,y'}F  \big\|_4 \\
&\qquad\times  \big\| D_{\theta,w} D_{s', y}F - D_{\theta,w} D_{s', y'}F - D_{\theta,w'} D_{s', y}F + D_{\theta,w'} D_{s', y'}F \big\|_4.
 \end{align*}
 We have the same bound for the first term \eqref{1st_term_in} except for the multiplicative constant $2$, due to $2   \int_0^\infty dt e^{-t} =2$. Hence, the proof of Proposition \ref{propA2} is completed.
 \end{proof}

\begin{remark}\label{rmk_a} 
\begin{enumerate}
\item[\rm(i)]  Compared to the \textbf{regular case}, the expression of $\mathcal{A}$ is much more complicated in the \textbf{rough case}, where we need  not only to control   $ \| D_{r,z} u(t,x) \|_p$, but we also  have to estimate  the more notorious differences $\| D_{r',z}u(t,x) -  D_{r',z'} u(t,x)\|_p$ and  $  \| D_{r,z}D_{s,y} u(t,x) - D_{r,z}D_{s,y'} u(t,x)  - D_{r,z'}D_{s,y} u(t,x)  + D_{r,z'}D_{s,y'} u(t,x)    \|_p$ for the proof of part (3) in Theorem \ref{thm_QCLT}. This is the current paper's highlight in regard of the technicality.

\item[\rm(ii)] When $\gamma_0$ is the Dirac delta function at zero, the expression of $\mathcal{A}$ reduces to 
\begin{align*}
\mathcal{A} &= \int_{\R_+^3} ds d\theta   dr   \int_{\R^6} dydy'dzdz'dwdw' |y-y'|^{2H_1-2}  |z-z'|^{2H_1-2} |w-w'|^{2H_1-2}\\
&\qquad\times  \| D_{r,z}F -  D_{r,z'} F\|_4 \| D_{\theta,w}F -  D_{\theta,w'} F\|_4 \\
&\qquad\times  \big\| D_{r,z}D_{s,y}F - D_{r,z}D_{s,y'}F - D_{r,z'}D_{s,y}F + D_{r,z'}D_{s,y'}F  \big\|_4 \\
&\qquad\times  \big\| D_{\theta,w} D_{s, y}F - D_{\theta,w} D_{s, y'}F - D_{\theta,w'} D_{s', y}F + D_{\theta,w'} D_{s, y'}F \big\|_4.
\end{align*}
This case  corresponds to the white-in-time setting where the driving noise $W$ behaves like Brownian motion in time, so  as an alternative to using Proposition \ref{propA2}, one may adapt the general strategy based on the Clark-Ocone formula $($see {\rm\cite{HNV20})} to establish the quantitative CLT for $F_R(t)$; however, the roughness in space will anyway force one to  use the Gagliardo formulation $($see {\rm\eqref{def_gag})} of the inner product on $\kH$ when estimating the variance of $\langle DF_R(t), V_{t,R} \rangle_\kH$. This will lead to almost the same level of difficulty as in bounding the expression $\mathcal{A}$, while our computations will be done for a broader range of temporal correlation structures that include the Dirac delta function $($white-in-time case$)$. 
\end{enumerate}

\end{remark}

\subsection{Technical lemmas}
In this subsection, we provide some useful results related to the heat kernel and gamma functions. All of them will be used in Section \ref{sec_rough}. Let us first introduce a few more notation.  Set
 \begin{align}\label{delta}
\Delta_t(x,x')=p_t(x+x')-p_t(x),
\end{align}
\begin{align}\label{r}
R_t(x,x',x'')=p_t(x+x'-x'')-p_t(x+x')-p_t(x-x'')+p_t(x)
\end{align}
and
\begin{align}\label{n}
N_t(x)=t^{\frac{1}{8} - \frac{1}{2}H_0}|x|^{H_0-\frac{1}{4}}\1_{\{|x|\leq \sqrt{t}\}}+\1_{\{|x|>\sqrt{t}\}},
\end{align}
for all $t\in\R_+$ and $x,x',x''\in \R$. The next lemma provides further estimates for $\Delta_t$ and  $R_t$. This lemma, as well as operator $\Lambda$ (see \eqref{def_lmd} below), will be used in Proposition \ref{corodu} combined with the simplified formulas  in Lemmas \ref{produ} and \ref{prod2u}.

\begin{lemma}\label{tl1}
Let $\Delta_t$, $R_t$ and $N_t$ be given as in \eqref{delta} - \eqref{n}. Then, the following results hold:
\begin{align}\label{ineq_nn}
\int_{\R}  \big[ N_t(x) \big]^2 |x|^{2H_1-2} dx = 4 \frac{1-H_1}{1-2H_1} t^{H_1-\frac{1}{2}},
\end{align}
\begin{align}\label{delta0}
|\Delta_t(x,x')|\leq c_{\beta} \big( \Phi_{t,x'}^{\beta}p_{4t} \big)(x)
\end{align}
and
\begin{align}\label{r0}
|R_t(x,x',x'')|\leq c_{\beta}  (\Phi_{t,x'}^{\beta}\Phi_{t,-x''}^{\beta}p_{4t})(x)
\end{align}
for any $\beta\in [0,1]$, $t\in \R_+$, $x,x',x''\in \R$ with some constant $c_{\beta}$ depending only on $\beta$, where $\Phi_{t,x'}^{\beta}$ is the operator acting on $\cM(\R)$, the real measurable functions on $\R$, given by
\[
(\Phi_{t,x'}^{\beta}g)(x)=\theta_{x'}g(x)\1_{\{|x'|>\sqrt{t}\}}+g(x)\Big[\1_{\{|x'|>\sqrt{t}\}}+t^{-\frac{\beta}{2}}|x'|^{\beta}\1_{\{|x'|\leq \sqrt{t}\}}\Big],
\]
with $\theta$ denoting the shift function, that is, $(\theta_{x'}g)(x)=g(x+x')$.
\end{lemma}

\begin{remark}\label{rem_0_1}  For any $t>0$ and $z\in\R$, operator $\Phi^\beta_{t, z}$ can be expressed as 
\[
\Phi^\beta_{t,z} g  = \1_{\{ | z| > \sqrt{t} \}  } \big( \theta_z + \mathbf{I} \big)g +  \1_{\{ | z| \leq \sqrt{t} \}  }  \big( |z| t^{-1/2} \big)^\beta \mathbf{I}g, \,\,\,\, g\in\cM(\R)
\]
with  $ \mathbf{I}  g = g$. It is easy to check that the following commutativity property holds
 $
\Phi^\beta_{t,z}    \Phi^\beta_{s,y}   = \Phi^\beta_{s,y}    \Phi^\beta_{t,z} 
$ on $\cM(\R)$. Furthermore, it is also clear that $\big(\Phi_{t,z}\1_\R\big)(0) \leq 2 N_t(z)$.
\end{remark}

\begin{proof}[Proof of Lemma \ref{tl1}]  Equality \eqref{ineq_nn} follows from direct computations, which we omit here.  In what follows, we first derive the estimate \eqref{delta0} for  $\Delta_t$. If $|x'|>\sqrt{t}$,  it follows immediately that 
\begin{align}\label{deltag}
 \big\vert \Delta_t(x,x') \big\vert \leq p_t(x+x')+p_t(x) \leq 2 \big[ p_{4t}(x+x')+p_{4t}(x)  \big].
\end{align}
On the other hand, suppose now that $|x'|\leq \sqrt{t}$. Notice firstly that
\[
|\Delta_t(x,x')| =|p_t(x+x')-p_t(x)| < \max\{ p_t(x), p_t(x+x') \}   \leq \frac{1}{\sqrt{2\pi}}t^{-\frac{1}{2}}.
\]
By the mean value theorem, there exists a point $z$ between $0$ and $x'$ such that,
\[
|\Delta_t(x,x')|=|p_t(x+x')-p_t(x)|  =|x'|\frac{|x+z|}{(2\pi)^{1/2}t^{3/2} }e^{-\frac{(x+z)^2}{2t}}.
\]
In view of the fact that for all $\alpha>0$,
\begin{align}\label{exp}
\sup_{x>0}x^{\alpha}e^{-x}=\alpha^{\alpha}e^{-\alpha},
\end{align}
we know that $\frac{|x+z|}{\sqrt{t}}e^{-\frac{(x+z)^2}{4t}}$ is uniformly bounded, from which it follows that
\[
|\Delta_{t}(x,x')|\leq   c_1 \frac{|x'|}{t}e^{-\frac{(x+z)^2}{4t}}.
\]
Since $\frac{(x+z)^2}{4t}\geq   \frac{x^2}{8t} - \frac{z^2}{4t}$ and $|z|\leq |x'|\leq \sqrt{t}$, we get 
\[
e^{-\frac{(x+z)^2}{4t}}\leq e^{\frac{z^2}{4t}}e^{-\frac{x^2}{8t}}\leq e^{\frac{1}{4}}e^{-\frac{x^2}{8t}},
\]
and thus
\[
|\Delta_t(x,x')|\leq   c\frac{|x'|}{\sqrt{t}}p_{4t}(x).
\]
for some universal constant $c>0$. As a consequence, if $|x'|\leq \sqrt{t}$, for any $\beta \in [0,1]$, we have
\begin{align}\label{deltal}
|\Delta_t(x,x')|\leq c_{\beta} t^{-\frac{\beta}{2}}|x'|^{\beta}p_{4t}(x),
\end{align}
where $c_{\beta}$ depends only on $\beta$. Putting together the estimates \eqref{deltag} and \eqref{deltal} yields  \eqref{delta0}.

\medskip

Next, we prove the inequality \eqref{r0} by considering  the following four cases.

\medskip

\textit{Case 1.} If $|x'|>\sqrt{t}$ and $|x''|>\sqrt{t}$,  we use the estimate
\begin{align}\label{r1}
|R_t(x,x',x'')|  \leq p_t(x+x'-x'')+p_t(x+x')+p_t(x-x'')+p_t(x).
\end{align}

\textit{Case 2.} If $|x'|\leq \sqrt{t}$ and $|x''|> \sqrt{t}$,    we deduce from  \eqref{deltal} that 
\begin{align}
|R_t(x,x',x'')|  \leq &|p_t(x+x'-x'')-p_t(x-x'')|+|p_t(x+x')+p_t(x)|\nonumber\\
=& |\Delta_t(x-x'',x')|+|\Delta_t(x,x')|    \nonumber\\
\leq & c_{\beta} t^{-\frac{\beta}{2}}|x'|^{\beta}\big(p_{4t}(x)+p_{4t}(x-x'')\big).
\end{align}

\textit{Case 3.} If  $|x'|> \sqrt{t}$ and $|x''|\leq \sqrt{t}$,  then  the same argument from  Case 2 leads to 
\begin{align}
|R_t(x,x',x'')|  \leq c_{\beta} t^{-\frac{\beta}{2}}|x''|^{\beta}\big(p_{4t}(x)+p_{4t}(x+x')\big).
\end{align}

\textit{Case 4.} If $|x'|\leq \sqrt{t}$ and $|x''|\leq \sqrt{t}$, we write 
\[
R_t(x,x',x'') = \int_{0}^{-x''} \big[ p'_t(x+x'+y) - p'_t(x+y) \big]dy
\]
with the convention $\int_0^{-a} f(t) dt = -\int_0^a f(-t)dt$ for $a > 0$.   By mean value theorem, we can write 
\begin{align*}
p'_t(x+x'+y) - p'_t(x+y) & = p_t''( x+ y + z') x' \\
 &= \frac{x'}{\sqrt{2\pi}}  e^{-\frac{(x+y+z')^2}{2t} }  \big( t^{-5/2} (x+y+z')^2 - t^{-3/2} \big),
\end{align*}
where $z'$ is some number between $0$ and $y$. Using  \eqref{exp} and $ (x+y+z')^2  \geq \frac{1}{2}x^2 - (y+z')^2$, we have
\begin{align*}
e^{-\frac{(x+y+z')^2}{2t} }    t^{-5/2} (x+y+z')^2  \leq c \, e^{-\frac{(x+y+z')^2}{4t} } t^{-3/2} \leq c\,  p_{4t}(x) t^{-1} e^{ \frac{(y+z')^2}{4t} }.
\end{align*}
Since $|y+z'| \leq 2|x''| \leq 2\sqrt{t}$, we get 
\[
e^{-\frac{(x+y+z')^2}{2t} }    t^{-5/2} (x+y+z')^2  \leq c\, p_{4t}(x) t^{-1} 
\]
and
\[
e^{-\frac{(x+y+z')^2}{2t} } t^{-3/2}  \leq c\, e^{ \frac{(y+z')^2}{2t} } \frac{p_{2t}(x)}{t} \leq c\,   \frac{p_{4t}(x)}{t}. 
\]
It follows that 
\begin{align}\label{r4}
|R_t(x,x',x'')|  \leq c_{\beta} t^{-\beta}|x'|^{\beta}|x''|^{\beta}p_{4t}(x),
\end{align}
provided   $|x'|\leq \sqrt{t}$ and $|x''|\leq \sqrt{t}$. Therefore, inequality \eqref{r0} follows from \eqref{r1} - \eqref{r4}. The proof of this lemma is completed.
\end{proof}
We also introduce the operator $\Lambda:\cM(\R)\times \cM(\R)\to \cM(\R^2)$ as follows. This operator  will be used in Section \ref{sec_rough}. Let $0<r<s$ and let $z',y'\in \R$. Then,
\begin{align}\label{def_lmd}
\Lambda_{r,z',s,y'}&(g_1,g_2)(x,y)=g_1(x)(\Phi_{s-r,y'}g_2)(y)N_r(z')+g_1(x)(\Phi_{s-r,y'}\Phi_{s-r,-z'}g_2)(y)\nonumber\\
&+(\Phi_{t-s,-y'}g_1)(x)(\theta_{y'}g_2)(y)N_r(z')+(\Phi_{t-s,-y'}g_1)(x)(\theta_{y'}\Phi_{s-r,-z'}g_2)(y),
\end{align}
for any $(g_1,g_2, x, y)\in \cM(\R)\times \cM(\R)\times\R^2$.  

\begin{remark}\label{rem_Lambda} It is not difficult to see that if $\| g_1\|_{L^1(\R)} =1$, then
\[
\int_\R \big( \Phi_{t,x'}g_1\big)(x)dx =  \big( \Phi_{t,x'} \1_\R\big)(x)
\]
with $\1_\R(x) =1$ for all $x\in\R$. As a result, we have 
\[
\int_{\R} dx \Lambda_{r,z',s,y'}(g_1,g_2)(x,y) =  \Lambda_{r,z',s,y'}(\1_\R,g_2)(0,y)
\]
provided $\| g_1\|_{L^1(\R)}=1$. 
\end{remark}

We complete this subsection by the following results about the gamma functions.
\begin{lemma}\label{lmm_gmm}
Let  
\[
\Gamma(x)=\int_{0}^{\infty} y^{x-1}e^{-y}dy, \quad  x>0
\]
be the usual  gamma function\footnote{This shall not be confusing with $\Gamma_t$ defined as in Lemma \ref{lem_redu}.}, then we have the following bounds. 
\begin{enumerate}[(i)]
\item[\rm(i)] {\rm(Stirling's formula; c.f. \cite[Theorem 1]{mg-15-jameson})} For all $x>0$, the following inequality holds,
\[
\sqrt{2\pi}x^{x-\frac{1}{2}}e^{-x}\leq \Gamma(x)\leq \sqrt{2\pi}x^{x-\frac{1}{2}}e^{-x+\frac{1}{12x}}.
\]
\item[\rm (ii)] {\rm(Asymptotic bound of the Mittag-Leffler function; c.f. \cite[Formula (1.8.10)]{elsevier-06-kilbas-strivastava-trujillo})}
\begin{align*}
E_{\alpha}(z)=\sum_{n=0}^{\infty}\frac{z^n}{\Gamma(\alpha n+1)}\leq c_1 \exp\big(c_2z^{\frac{1}{\alpha}}\big)
\end{align*}
for all $z\in \R_+$ and $\alpha\in (0,2)$, where $c_1,c_2>0$ depend only on $\alpha$. 
  \end{enumerate}
\end{lemma}
As a result of Lemma \ref{lmm_gmm} (i), we can deduce the following corollary.
\begin{corollary}\label{coro_gmm}
Let $\alpha_1,\alpha_2,\beta_1,\beta_2\in [0,\infty)$, and let $\gamma_1,\gamma_2\in \R$. Then, the following inequality holds for $n$ large enough
\begin{align}\label{coro_gmm1}
\Gamma(\alpha_1 n+\gamma_1)^{\beta_1}\Gamma(\alpha_2n+\gamma_2)^{\beta_2}\leq c_1c_2^n\Gamma[(\alpha_1\beta_1+\alpha_2\beta_2)n+1].
\end{align}
Moreover, if $\alpha_1\beta_1<\alpha_2\beta_2$, then, for $n$ large enough,
\begin{align}\label{coro_gmm2}
\frac{\Gamma(\alpha_1 n+\gamma_1)^{\beta_1}}{\Gamma(\alpha_2n+\gamma_2)^{\beta_2}}\leq \frac{c_1c_2^n}{\Gamma[(\alpha_2\beta_2-\alpha_1\beta_1)n+1]}.
\end{align}
In both inequalities \eqref{coro_gmm1} and \eqref{coro_gmm2},  $c_1,c_2>0$ are independent of $n$.
\end{corollary}
\begin{proof}
We only show inequality \eqref{coro_gmm1}. Inequality \eqref{coro_gmm2} can be proved in a similar way. If $\beta_1=\beta_2=0$, then \eqref{coro_gmm1} is trivial. Assume that $\beta_1+\beta_2>0$. Using Lemma \ref{lmm_gmm} (i), we have
\begin{align*}
{\rm LHS}:=\Gamma(\alpha_1 n+\gamma_1)^{\beta_1}\Gamma(\alpha_2n+\gamma_2)^{\beta_2}\leq &\big(\sqrt{2\pi}(\alpha_1 n+\gamma_1)^{\alpha_1 n+\gamma_1-\frac{1}{2}}e^{-\alpha_1 n+\gamma_1+\frac{1}{12(\alpha_1 n+\gamma_1)}}\big)^{\beta_1}\\
&\times \big(\sqrt{2\pi}(\alpha_2 n+\gamma_2)^{\alpha_2 n+\gamma_2-\frac{1}{2}}e^{-\alpha_2 n+\gamma_2+\frac{1}{12(\alpha_2 n+\gamma_2)}}\big)^{\beta_2}.
\end{align*}
Choose $n\geq \max\{|\frac{\gamma_1}{\alpha_1}|+1,|\frac{\gamma_2}{\alpha_2}|+1\}$, then we have $1\leq \alpha_1 n+\gamma_1\leq 2\alpha_1 n$ and $\alpha_2 n+\gamma_2\leq 2\alpha_2 n$. This implies that
\begin{align}\label{coro_gmm11}
{\rm LHS}\leq &(2\pi)^{\frac{1}{2}(\beta_1+\beta_2)}e^{(\gamma_1+1)\beta_1+(\gamma_2+1)\beta_2} (2\alpha_1)^{(\alpha_1 n+\gamma_1-\frac{1}{2})\beta_1}(2\alpha_2)^{(\alpha_2 n+\gamma_2-\frac{1}{2})\beta_2}\nonumber\\
&\times n^{(\alpha_1\beta_1+\alpha_2\beta_2)n+(\gamma_1-\frac{1}{2})\beta_1+(\gamma_2-\frac{1}{2})\beta_2}e^{-(\alpha_1\beta_1+\alpha_2\beta_2)n}\nonumber\\
\leq &c_1c_2^n n^{(\gamma_1-\frac{1}{2})\beta_1+(\gamma_2-\frac{1}{2})\beta_2+\frac{1}{2}} [(\alpha_1\beta_1+\alpha_2\beta_2)n]^{(\alpha_1\beta_1+\alpha_2\beta_2)n-\frac{1}{2}}e^{-(\alpha_1\beta_1+\alpha_2\beta_2)n}.
\end{align}
As $\lim_{n\uparrow \infty}\frac{n^{\alpha}}{2^n}=0$ for all $\alpha\in \R$, inequality \eqref{coro_gmm1} is  then consequence of \eqref{coro_gmm11} and Lemma \ref{lmm_gmm} (i) for $n$ large enough. The proof of this corollary is completed.
\end{proof}

\section{Regular cases under Hypothesis \ref{h1}} \label{sec_regular}

In this section, we prove the first two error bounds in \eqref{main_TV}. As already mentioned in the introduction, the majority of this section will devoted to proving the following $L^p$ estimates of Malliavin derivatives.

\begin{theorem} \label{M_D_regular}
{\rm Assume that Hypothesis \ref{h1} holds true. Given $(t,x)\in (0,\infty)\times\R^d$ and $(p,m)\in [2,\infty)\times\N^\ast$,  then for almost every $(\bs_m, \by_m) \in ([0,t]\times\R^d)^m$, we have
\begin{align}\label{M_D_reg_p}
 \big\| D^m_{\bs_m,\by_m}u(t,x) \big\|_p \leq C(t) f_{t,x,m}\big(\bs_m,\by_m\big),
\end{align}
where $f_{t,x,m}$ is the chaos coefficient defined as in \eqref{cos2} and the constant $C(t)$ depends on $(t,p,m, \gamma_0,\gamma_1)$ and is increasing in $t$.
}
\end{theorem}
The proof of Theorem \ref{M_D_regular} is deferred to Section \ref{sec32_MD}. In Section \ref{sec31_reg}, we prove the  first two error bounds in \eqref{main_TV} by using Theorem \ref{M_D_regular}.

\subsection{Proof of quantitative CLTs in regular cases} \label{sec31_reg}
 Assume Hypothesis \ref{h1}. With 
\[
F_R(t) = \int_{B_R} [ u(t,x) - 1 ] dx \quad\text{and}\quad  \sigma_R(t) = \sqrt{ \text{Var}\big( F_R(t) \big)},
\]
we have the following facts from \cite{NZ19BM}.
\begin{enumerate}[(i)]
\item Under Hypothesis \ref{h3a},   $\sigma_R(t) \sim R^{d/2}$.
\item Under Hypothesis \ref{h3b},  one has  $\sigma_R(t) \sim R^{d- \frac{\beta}{2} }$.  
\end{enumerate}

\subsubsection{Proof of \texorpdfstring{\eqref{main_TV}}{} under Hypotheses \texorpdfstring{\ref{h1} and \ref{h3a}}{} } \label{sec_311} Using   Minkowski's inequality, we have
\[
\| D_{r,z}D_{s,y}F_R(t) \| _4 = \left\| \int_{B_R}  D_{r,z}D_{s,y} u(t,x) dx \right\|_4 \leq  \int_{B_R}  \big\| D_{r,z}D_{s,y} u(t,x) \big\|_4  dx.  
\]
Then it follows from   \eqref{M_D_reg_p} that
\begin{align}\label{USE1}
\| D_{r,z}D_{s,y}F_R(t) \| _4 \lesssim  \int_{B_R}  f_{t,x,2}(r,z, s,y) dx,
 \end{align}
 with
 \[
f_{t,x,2}(r,z,s,y) = \frac{1}{2} \left[ p_{t-r}(x-z) p_{r-s}(z-w)\mathbf{1}_{\{ r > s\}} + p_{t-s}(x-y) p_{s-r}(z-y)\mathbf{1}_{\{ r < s\}}\right].
 \]
In the same way, we have
\begin{align}\label{USE2}
\| D_{s,y}F_R(t) \| _4 \lesssim  \int_{B_R}  p_{t-s}(x-y)dx,
 \end{align}
 where the implicit constants in \eqref{USE1} and \eqref{USE2} do not depend on $(R, r,z,s,y)$ and are increasing in $t$.

Apply Proposition \ref{propA1} and plugging \eqref{USE1} and \eqref{USE2} into the expression of $\mathcal{A}$,  we get
 \begin{align*}
\mathcal{A}&\lesssim  \int_{[0,t]^6\times\R^{6d}} drdr'dsds' d\theta d\theta' dzdz' dydy' dwdw' \gamma_0(r-r') \gamma_0(s-s') \gamma_0(\theta - \theta')\gamma_1(z-z')     \\
&  \quad\times \int_{B_R^4}d\bx_4 \gamma_1(w-w')  \gamma_1(y-y') f_{t,x_1,2}(r,z, \theta, w)  f_{t,x_2,2}(s,y, \theta', w') p_{t-r'}(x_3- z')\\
&\qquad\times  p_{t-s'}(x_4 - y') .
  \end{align*}
 Taking the expression of $ f_{t,x,2}$ into account, we need to consider four terms depending on $r>\theta$ or not, and depending on $ s>\theta'$ or not. Since the computations are similar, it suffices to provide the estimate for case $r>\theta$ and $s>\theta'$. In other words, we need to show that 
  \begin{align*}
\mathcal{A}^\ast:&=  \int_{[0,t]^6\times\R^{6d}} drdr'dsds' d\theta d\theta' dzdz' dydy' dwdw' \gamma_0(r-r') \gamma_0(s-s') \gamma_0(\theta - \theta') \\
& \quad   \times \int_{B_R^4} d\bx_4 p_{t-r}(x_1-z) p_{r-\theta}(z-w)  p_{t-s}(x_2-y) p_{s-\theta'}(y-w')  p_{t-r'}(x_3- z') \\
&  \qquad \times     p_{t-s'}(x_4 - y') \gamma_1(w-w')  \gamma_1(y-y')  \gamma_1(z-z')  \\
&   \lesssim R^d.
  \end{align*}
 In fact, the above estimate follows from integrating with respect to $dx_1$, $dx_2$, $dx_4$, $dy'$, $dy$, $dw'$, $dw$, $dz$, $dz'$, $dx_3$ one by one and  using the local integrability of $\gamma_0$. The desired bound follows immediately. \hfill  $\square$

\subsubsection{Proof of \texorpdfstring{\eqref{main_TV}}{} under Hypotheses \texorpdfstring{\ref{h1} and \ref{h3b}}{} }  Similarly as in   Section \ref{sec_311}, we need to show $\mathcal{A}^\ast\lesssim R^{4d-3\beta}$. Making the change of variables
 \[
 (\bx_4, z,z', y, y', w, w') \to  R(\bx_4, z,z', y, y', w, w') 
 \]
and using the scaling properties of the Riesz and heat kernels\footnote{$p_t(Rz) = R^{-d} p_{tR^{-2}}(z)$ for $z\in\R^d$.} yields
  \begin{align*}
\mathcal{A}^\ast&=R^{4d-3\beta}   \int_{[0,t]^6} drdr'dsds' d\theta d\theta'  \gamma_0(r-r') \gamma_0(s-s') \gamma_0(\theta - \theta')\mathbf{S}_R,
  \end{align*}
  with
  \begin{align*}
  \mathbf{S}_R:&= \int_{B_1^4\times\R^{6d}} d\bx_4 dzdz' dydy' dwdw' | w-w'|^{-\beta}  | y-y'|^{-\beta}  |z-z'|^{-\beta}  \\
   &\quad \times p_{\frac{t-r}{R^2}}(x_1-z) p_{\frac{r-\theta}{R^2}}(z-w)  p_{\frac{t-s}{R^2}}(x_2-y) p_{\frac{s-\theta'}{R^2}}(y-w')  p_{\frac{t-r'}{R^2}}(x_3- z') p_{\frac{t-s'}{R^2}}(x_4 - y').
\end{align*} 
Making the following change of variables 
\[
\pmb{\eta}_6 = (z-x_1, z-w, y-x_2, y-w', z'-x_3, y'-x_4)
\]
(so $w = \eta_1 - \eta_2 + x_1, w' = \eta_3-\eta_4+x_2$) yields
\begin{align*}
  \mathbf{S}_R &= \int_{B_1^4\times\R^{6d}} d\bx_4 d\pmb{\eta}_6  p_{\frac{t-r}{R^2}}(\eta_1) p_{\frac{r-\theta}{R^2}}(\eta_2)  p_{\frac{t-s}{R^2}}(\eta_3) p_{\frac{s-\theta'}{R^2}}(\eta_4)  p_{\frac{t-r'}{R^2}}(\eta_5) p_{\frac{t-s'}{R^2}}(\eta_6) \\
  &\quad \times |  \eta_1 - \eta_2 -\eta_3 +\eta_4 + x_1-x_2   |^{-\beta}  | \eta_3- \eta_6 +x_2-x_4 |^{-\beta}  | \eta_1 - \eta_5 + x_1 - x_3|^{-\beta}  \\
  &=  \int_{B_1^4} d\bx_4 \E\bigg[  \Big|  \frac{\sqrt{t-r}}{R} Z_1 - \frac{\sqrt{r-\theta}}{R}  Z_2 - \frac{\sqrt{t-s}}{R}    Z_3 +  \frac{\sqrt{s-\theta'}}{R}Z_4 + x_1-x_2   \Big|^{-\beta} \\
  &\qquad \times  \Big| \frac{\sqrt{t-s}}{R}    Z_3-  \frac{\sqrt{t-s'}}{R} Z_6 +x_2-x_4 \Big|^{-\beta}  \Big|  \frac{\sqrt{t-r}}{R} Z_1 - \frac{\sqrt{t-r'}}{R}Z_5 + x_1 - x_3  \Big|^{-\beta}             \bigg],
    \end{align*}
where $Z_1, \dots ,Z_6$ are i.i.d. standard Gaussian vectors on $\R^d$. Notice that
\begin{align*}
\mathcal{K}:= &\sup_{z\in\R^d} \int_{B_1} |y +z|^{-\beta} dy \leq \mathbf{1}_{\{ |z| \leq 2\}}  \int_{B_3} |y|^{-\beta} dy +  \mathbf{1}_{\{ |z| > 2\}}  \int_{B_1} \big(  |z| -|y|\big) ^{-\beta} dy \\
\leq &\int_{B_3} |y|^{-\beta} dy +  \int_{B_1} 1dy<\infty.
\end{align*}
Therefore, we deduce that
  \begin{align*}
  \mathbf{S}_R &\leq \mathcal{K}^3  \int_{B_1} 1~ dx_1 =  \mathcal{K}^3 \text{Vol}(B_1).
    \end{align*}
 So $\mathcal{A}^\ast\leq  \mathcal{K}^3   \text{Vol}(B_1)  (t\Gamma_t)^3 R^{4d-3\beta}$, with $\Gamma_t  = \int_{-t}^t \gamma_0(s)ds$.
Hence applying Proposition \ref{propA1} yields the desired conclusion. \hfill  $\square$

\subsection{Proof of Theorem \ref{M_D_regular}} \label{sec32_MD}

Recall the Wiener chaos expansion  \eqref{cos1} and \eqref{cos2} for $u(t,x)$. Then, for any positive integer $m$, the $m$-th Malliavin derivative valued at $(s_1, y_1, \dots , s_m, y_m)$ is given by 
  \[
  D^m_{\bs_m, \by_m } u(t,x)  = D_{s_1, y_1} D_{s_2, y_2} \cdots D_{s_m, y_m} u(t,x)  = \sum_{n= m}^{\infty} \frac{n!}{(n-m)!} I_{n-m}\big(  f_{t,x,n}(\bs_m, \by_m,\bullet ) \big),
  \]
 whenever the series converges in $L^2$. By definition, it is easy to check that $f_{t,x,n}(\bs_m, \by_m,\bullet )\in \kH^{\odot (n-m)}$ and by symmetry again, we can assume $t> s_m > s_{m-1}> \dots > s_1 > 0$. 
 
 \medskip
 
For any $n\geq m$, we make use of the notation, 
\[
[n]_{<}=\{\mathbf{i}_m=(i_1,\dots, i_m), 1\leq i_1< \dots <  i_m \leq n\}.
\]
We also define the function $f^{(\pmb{i}_m)}_{t,x,n}(\bs_m, \by_m; \bullet):\R_+^{n-m}\times (\R^d)^{n-m} $  by
\begin{align}\label{decomp-fjj}
f^{(\pmb{i}_m)}_{t,x,n}(\bs_m, \by_m; \bullet) =&    f^{(i_1)}_{t,x,i_1}(s_m, y_m; \bullet) \otimes  f^{(i_2-i_1)}_{s_m,y_m,i_2-i_1}(s_{m-1}, y_{m-1}; \bullet) \notag\\
 &\otimes \cdots \otimes  f^{(i_m-i_{m-1})}_{s_{2},y_{2}, i_m- i_{m-1}}(s_1, y_1; \bullet) \otimes   f_{s_1, y_1, n-i_m},  
\end{align}
where $f_{t,x,1}^{(1)}(r,z; \bullet) = p_{t-r}(x-z)$, and for all $k\geq 2$,
\begin{align*} 
f_{t,x,k}^{(k)}(r,z;  \bs_{k-1}, \by_{k-1}) :=&\frac{1}{k!} p_{t-s_{\sigma(k-1)}}(x-y_{\sigma(k-1)})p_{s_{\sigma(k-1)}-s_{\sigma(k-2)}}(y_{\sigma(k-1)}- y_{\sigma(k-2)})\\
&\times \cdots \times p_{s_{\sigma((1)} - r}(y_{\sigma(1)} -z), \notag
\end{align*}
and $\sigma$ denotes the permutation of $\{1,\dots,k-1\}$ such that $r<s_{\sigma(1)}<\dots<s_{\sigma(k-1)}<t$. Let $h^{(\pmb{i}_m)}_{t,x,n}(\bs_m, \by_m; \bullet)$ be the symmetrization of $f^{(\pmb{i}_m)}_{t,x,n}(\bs_m, \by_m; \bullet)$. Then, for any  $p\in[2,\infty)$, we deduce from Minkowski's inequality and    \eqref{hyper_ineq} that 
  \begin{align*}
\big\|   D^m_{\bs_m, \by_m } u(t,x)  \big\|_p &\leq  \sum_{n= m}^{\infty} (p-1)^{\frac{n-m}{2}}   \Big\|  I_{n-m}\Big( \frac{n!}{(n-m)!} f_{t,x,n}( \bs_m, \by_m,\bullet  ) \Big) \Big\|_2 
  \end{align*}
and
\begin{align}
 \frac{n!}{(n-m)!}   f_{t,x,n}(\bs_m, \by_m, \bullet) = \sum_{\pmb{i}_m\in [n]_{<}} h^{(\pmb{i}_m)}_{t,x,n}(\bs_m, \by_m; \bullet).  \notag 
\end{align}
It follows that
  \begin{equation}  \label{EQ1}
   \Big\|  I_{n-m}\Big( \frac{n!}{(n-m)!} f_{t,x,n}( \bs_m, \by_m,\bullet  ) \Big) \Big\|_2^2 \leq \binom{n}{m} \sum_{\pmb{i}_m\in [n]_{<}} \big\| I_{n-m}\big( f^{(\pmb{i}_m)}_{t,x,n}(\bs_m, \by_m; \bullet)    \big) \big\|_2^2.
   \end{equation}
Additionally, due to Lemma \ref{lem_redu}, we  deduce that 
  \begin{equation}  \label{EQ2}
\| I_{n-m}( f^{(\pmb{i}_m)}_{t,x,n}(\bs_m, \by_m; \bullet)    ) \|_2^2\leq \Gamma_t^{n-m} \| I_{n-m}^\X( f^{(\pmb{i}_m)}_{t,x,n}(\bs_m, \by_m; \bullet)    ) \|_2^2.
   \end{equation}
   The inequalities  \eqref{EQ1} and \eqref{EQ2} together with the product formula  given in Lemma \eqref{lem_prod} and the decomposition  \eqref{decomp-fjj},  imply
   \begin{align} \label{Q_H12}
\big\|   D^m_{\bs_m, \by_m } u(t,x)  \big\|_p &\leq  \sum_{n= m}^{\infty} \big[ (p-1) \Gamma_t\big]^{\frac{n-m}{2}}  \sqrt{\mathcal{Q}^\X_{n,m}} ,   
   \end{align}
   where
   \begin{align*}
\mathcal{Q}^\X_{n,m}  & =     \binom{n}{m}   \sum_{\pmb{i}_m\in [n]_{<}}  \big\| I_{n-m}  \big( f^{(\pmb{i}_m)}_{t,x,n}(\bs_m, \by_m; \bullet)    \big) \big\|_2^2.
\end{align*}
Using the  independence among the random variables inside the expectation, see Lemma \ref{lem_prod},  and the notation $(i_0, s_{m+1}, y_{m+1})=(0, t,x)$, we can write
\begin{align*}
\big\| I_{n-m} \big( f^{(\pmb{i}_m)}_{t,x,n}(\bs_m, \by_m; \bullet)    \big) \big\|_2^2 &=
 \big \| I^{\mathfrak{X}} \big( f_{s_1, y_1, n-i_m} \big)\big\|^2_2  \\
 &  \times  \prod_{j=1}^m\Big\| I^{\mathfrak{X}}_{i_j - i_{j-1}-1} \Big( f^{ (i_j- i_{j-1})  }_{s_{m-j+2}, y_{m-j+2}, i_j-i_{j-1}}(s_{m-j+1}, y_{m-j+1};\bullet) \Big) \Big\|^2_2.
     \end{align*}
 Thanks to the isometry property between the space $\kH^{\odot n}$ (see \eqref{def_hdn}), equipped with the modified norm $\sqrt{n!}\|\bullet\|_{\kH^{\otimes n}}$, and the $n$-th Wiener chaos $\mathbf{H}_n$, we can write
  \begin{align} \label{so_far}
  \mathcal{Q}^\X_{n,m} = & \binom{n}{m}   \sum_{\pmb{i}_m\in [n]_{<}}   (n-i_m)! \big\|  f_{s_1, y_1, n-i_m}  \big\|^2_{\kX^{\otimes (n-i_m)}   }   \notag \\
   & \times   \prod_{j=1}^m  (i_j-i_{j-1}-1)!\big\|   f^{ (i_j- i_{j-1})  }_{s_{m-j+2}, y_{m-j+2}, i_j-i_{j-1}}(s_{m-j+1}, y_{m-j+1};\bullet)  \big\|^2_{\kX^{\otimes (i_j - i_{j-1})}}. 
  \end{align}

  We first estimate  $ \big\| f_{s,y,k} \big\|_{\kX^{\otimes k}}$ and begin with
   \begin{align}
    \big\| f_{s,y,k} \big\|_{\kX^{\otimes k}}^2 & = n!\int_{\dT_k^s}d \br_k \big\|  f_{s,y,k}(\br_k, \bullet) \big\|_{\cH_1^{\otimes k}}^2 \notag \\
    & = \frac{1}{n!}\int_{\dT_k^s} d \br_k \int_{\R^{dk}}\mu^{\otimes k}(d\pmb{\xi}_k) \prod_{j=1}^{k} \big\vert \widehat{p}_{ r_{j+1} - r_j} (\xi_j +\dots +\xi_k )\big\vert^2,      \label{TAP_0}
   \end{align}
   with $r_{k+1}=s$ and $\widehat{p}_{t}(\xi)  = e^{-t |\xi|^2/2}$.     In the current \textbf{regular case}, we can deduce from the maximal principle (c.f. \cite[Lemma 4.1]{NZ19BM}) that 
  \[
  \sup_{z\in\R^d} \int_{\R^d}\mu(dy) e^{-s |y +z|^2}   =  \int_{\R^d}\mu(d y) e^{-s |y |^2}  <\infty.
  \]
 Then, preforming change of variables $w_{j}=r_{j+1}-r_j$, and using Lemma 3.3 in \cite{HHNT15}, we have
  \begin{align}\label{lem33_HHNT}
   \big\| f_{s,y,k} \big\|_{\kX^{\otimes k}}^2   \leq  &\frac{1}{k!} \int_{\mathcal{T}_k(s)}   d\bw_k   \int_{\R^{dk}} \mu(d\pmb{\xi}_k) \prod_{j=1}^k e^{-w_j  |\xi_j|^2} \leq \frac{1}{k!}\sum_{\ell =0}^k\binom{k}{\ell}\frac{s^{\ell}}{\ell !}D_N^{\ell}(2C_N)^{k-\ell}\nonumber\\
   \leq & \frac{2^k}{k!}\sum_{\ell =0}^k\frac{s^{\ell}}{\ell !}D_N^{\ell}(2C_N)^{k-\ell},
  \end{align}
  where $\mathcal{T}_k(s): = \{ \bw_k\in\R_+^k: w_1 + \dots + w_k \leq s\}$ and
  \begin{align}\label{CD_N}
  C_N : = \int_{\R^d} \frac{\mu(d\xi)}{ | \xi|^2 } \1_{\{ | \xi | \geq N \}} \quad{\rm and}\quad D_N : = \mu\big( \{ \xi\in\R^d:  | \xi | \leq N  \} \big),
  \end{align}
 are finite quantities under Dalang's condition \eqref{D_cond}. 

  Finally, in what follows, we estimate  $  \big\|  f_{s, y, k}^{(k)}(r,z;\bullet)  \big\|_{ \kX^{\otimes (k-1)}  }  $. It is trivial that for $k=1$,  
\begin{align}
 \big\|  f_{s, y, 1}^{(1)}(r,z;\bullet)  \big\|_{ \kX^{\otimes 0}  }^2 = p_{s-r}(y-z)^2. \label{quad_1}
  \end{align}
For $k=2$, we can write
\begin{align*}
 \big\|  f_{s, y, 2}^{(2)}(r,z;\bullet)  \big\|_{ \kX  }^2  &= \frac{1}{4}\int_r^s dv   \big\| p_{s-v}(y-\bullet) p_{v-r}(\bullet -z)  \big\|_{\cH_1}^2    \\
 &= \frac{1}{4}p_{s-r}(y-z)^2 \int_r^s dv  \big\| \frac{p_{s-v}(y-\bullet) p_{v-r}(\bullet -z)}{ p_{s-r}(y-z)  }  \big\|_{\cH_1}^2. 
 \end{align*}
 Using the fact (c.f. \cite[Formula (1.4)]{CKNP20-2}) that
\begin{align}\label{tilt}
\frac{p_t(a) p_s(b)}{p_{t+s}(a+b)} = p_{st/(t+s)}\Big(b - \frac{s}{s+t} (a+b) \Big),
\end{align}
we get 
\[
f(x):=\frac{p_{s-v}(y-x) p_{v-r}(x -z)}{ p_{s-r}(y-z)  } = p_{(v-r)(s-v)/(s-r)  } \Big(x-z - \frac{v-r}{s-r} (y-z) \Big).
\]
The Fourier transform  of $f$ is given by 
\[
\widehat{f}(\xi)=\exp\Big(-i  \Big (z + \frac{v-r}{s-r} (y-z) \Big)  \xi\Big) \exp\left(- \frac{(v-r)(s-v)}{2(s-r) } |\xi|^2  \right).
\]
This implies that 
\begin{align*}
\Big\| \frac{p_{s-v}(y-\bullet) p_{v-r}(\bullet -z)}{ p_{s-r}(y-z)  }  \Big\|_{\cH_1}^2 =c_{H_1} \int_{\R^d}\mu(d\xi) \exp\left(   -\frac{(v-r)(s-v)}{s-r}  |\xi|^2   \right),
\end{align*}
 and thus 
 \begin{align}
 \big\|  f_{s, y, 2}^{(2)}(r,z;\bullet)  \big\|_{ \kX  }^2  &=  c_{H_1}\frac{1}{4} p^2_{s-r}(y-z) \int_r^s dv   \int_{\R^d}\mu(d\xi) \exp\left(   -\frac{(v-r)(s-v)}{s-r}  |\xi|^2   \right) \notag  \\
 & \leq \frac{1}{4} c_{H_1}  p_{s-r}(y-z)^2 \big[ (s-r)D_N +    4C_N \big], \label{bdd_fkk22}
  \end{align}
  for any $N>0$.       Indeed, 
\begin{align*}
 & \int_r^s dv   \int_{\R^d}\mu(d\xi) \exp\left(   -\frac{(v-r)(s-v)}{s-r}  |\xi|^2   \right)  \\
 &\leq (s-r)D_N +    \int_{| \xi | \geq N}\mu(d\xi)   \int_r^s dv \exp\left(   -\frac{(v-r)(s-v)}{s-r}  |\xi|^2   \right)
 \end{align*}
and 
\begin{align*}
 &\quad  \int_r^s dv \exp\left(   -\frac{(v-r)(s-v)}{s-r}  |\xi|^2   \right) =  \int_0^{s-r} dv \exp\left(   -\frac{v(s-r-v)}{s-r}  |\xi|^2   \right) \\
 &=(s-r) \int_0^{1} dv e^{   -v(1-v) (s-r)  |\xi|^2 } = 2(s-r) \int_0^{1/2} dv e^{   -v(1-v) (s-r)  |\xi|^2 } \\
 &\leq 2(s-r) \int_0^{1/2} dv e^{   -v  \frac{ (s-r)  |\xi|^2}{2} } = \frac{4}{|\xi|^2} \left(1 - e^{ -(s-r) |\xi|^2/4 } \right) \leq  \frac{4}{|\xi|^2},
 \end{align*}
 so that 
 \begin{align*}
 &\quad \int_r^s dv   \int_{\R^d}\mu(d\xi) \exp\left(   -\frac{(v-r)(s-v)}{s-r}  |\xi|^2   \right) \\
 &= (s-r) \int_0^1 dv \int_{\R^d}\mu(d\xi) \exp\left(   -(s-r)v(1-v) |\xi|^2   \right) \leq (s-r)D_N +    4C_N, 
   \end{align*}
 for all $N>0$.

For $k\geq 3$, we obtain
\begin{align*}
  \big\|  f_{s, y, k}^{(k)}(r,z;\bullet)  \big\|_{ \kX^{\otimes (k-1)}}^2  =n! \int_{\dT_{k-1}^{r,s}} d\br_{k-1} \big\| f_{s, y, k}^{(k)}(r,z; \br_{k-1} , \bullet )  \big\|^2_{\cH_1^{\otimes k-1} } 
\end{align*}
In order to estimate $\| f_{s, y, k}^{(k)}(r,z; \br_{k-1} , \bullet )  \||_{\cH_1^{\otimes k-1} }$, we apply formula \eqref{tilt} several times and get another expression for $ f_{s, y, k}^{(k)}(r,z; \br_{k-1} , \bz_{k-1} )$ as follows,
\begin{align}\label{ine_cdt}
f_{s, y, k}^{(k)}&(r,z; \br_{k-1} , \pmb{z}_{k-1})=\frac{1}{k!}p_{s-r}(y-z)\frac{p_{s-r_{k-1}}(y-z_{k-1})p_{r_{k-1}-r_{k-2}}(z_{k-1}-z_{k-2})}{p_{s-r_{k-2}}(y-z_{k-2})}\nonumber\\
&\times\frac{p_{s-r_{k-2}}(y-z_{k-2})p_{r_{k-2}-r_{k-3}}(z_{k-2}-z_{k-3})}{p_{s-r_{k-3}}(y-z_{k-3})}\times\cdots\times \frac{p_{s-r_{k+1}}(y-z_{1})p_{r_{1}-r}(z_{1}-z)}{p_{s-r}(y-z)}\nonumber\\
=&p_{s-r}(y-z)\prod_{i=1}^{k-1}p_{\frac{(s-r_{i})(r_{i}-r_{i-1})}{s-r_{i-1}}}\Big(z_{i} - \frac{s-r_{i}}{s-r_{i-1}}z_{i-1} - \frac{r_{i}-r_{i-1}}{s-r_{i-1}}y\Big),
\end{align}
where by convention, $r_0=r$ and $z_0=z$. In the next step, we compute the Fourier transform of $f_{s, y, k}^{(k)}(r,z; \br_{k-1} , \bullet ) $. Using the representation \eqref{ine_cdt}, we integrate the following expression
$f_{s, y, k}^{(k)}(r,z; \br_{k-1} , \bz_{k-1} )e^{i\lgl \pmb{\xi}_{k-1}, \bz_{k-1}\rgl }$ subsequently in $\xi_{k-1},\xi_{k-2},\dots, \xi_1$ and get
\begin{align}\label{ine_fcdt}
\hf_{s, y, k}^{(k)}(r,z&; \br_{k-1}, \bxi_{k-1})=\frac{1}{k!}p_{s-r}(y-z)\exp\Big(-iz\frac{s-r_{1}}{s-r}\sum_{k_1=1}^{k-1}\Big(\xi_{k_1}\prod_{k_2=2}^{k_1}\frac{s-r_{k_2}}{s-r_{k_2-1}}\Big)\Big)\nonumber\\
&\times \exp\bigg[-iy\sum_{j=1}^{k-1}\frac{r_{j}-r_{j-1}}{s-r_{j-1}}\sum_{k_1=j}^{k-1}\Big(\xi_{k_1}\prod_{k_2=j+1}^{k_1}\frac{s-r_{k_2}}{s-r_{k_2-1}}\Big)\bigg]\nonumber\\
&\times \prod_{j=1}^{k-1}\exp\bigg[-\frac{(s-r_j)(r_j-r_{j-1})}{2(s-r_{j-1})}\Big[\sum_{k_1=j}^{k-1}\Big(\xi_{k_1}\prod_{k_2=j+1}^{k_1}\frac{s-r_{k_2}}{s-r_{k_2-1}}\Big)\Big]^2\bigg].
\end{align}
Notice that
\[
\sum_{k_1=j}^{k-1}\Big(\xi_{k_1}\prod_{k_2=j+1}^{k_1}\frac{s-r_{k_2}}{s-r_{k_2-1}}\Big)=\xi_j+\sum_{k_1=j+1}^{k-1}\Big(\xi_{k_1}\prod_{k_2=j+1}^{k_1}\frac{s-r_{k_2}}{s-r_{k_2-1}}\Big).
\]
 Thus, by the  the maximal principle (\cite[Lemma 4.1]{NZ19BM}) again, we have 
 \begin{align*}
 \big\| f_{s, y, k}^{(k)}(r,z;  \bullet )  \big\|^2_{\kX^{\otimes (k-1)}}=&k!\int_{\dT_{k-1}^{r,s}}d\br_{k-1}\int_{\R^{k-1}}\mu^{\otimes (k-1)}(d\bxi_{k-1})\big|\hf_{s, y, k}^{(k)}(r,z; \br_{k-1}, \bxi_{k-1})\big|^2\\
  \leq & \frac{1}{k!} p^2_{s-r}(y-z)    \mathcal{I},
    \end{align*}
  where, by the  change of variables $v_{i}=\frac{s-r_{k-i}}{s-r}$ for all $i=1,\dots, k$, 
  \begin{align*}
  \mathcal{I}=    &\int_{\dT_{k-1}^{r,s}}  d\br_{k-1}  \prod_{j=1}^{k-1}\int_{\R^{d} } \mu(d\xi_j) \exp\Big(-\frac{(s-r_j)(r_j-r_{j-1})}{(s-r_{j-1})}|\xi_{j}|^2\Big)\\
  =&(s-r)^{k-1}\int_{\dT_{k-1}^{1}}  d\bv_{k-1}  \int_{\R^{(k-1)d} } \mu^{\otimes (k-1)}(d\pmb{\xi}_{k-1}) \prod_{j=1}^{k-1}\exp\Big(-\frac{(s-r)v_j(v_{j+1}-r_{j})}{v_{j}}|\xi_{j}|^2\Big).
  \end{align*}
  We  estimate the term $\cI$ as follows.  First we make the decomposition
  \begin{align*}
     \mathcal{I} =&\sum_{J\subset\{1, \dots, k-1\}} \int_{\dT_{k-1}^1}  d\bv_{k-1}  \int_{\R^{(k-1)d} }\mu^{\otimes (k-1)}(d\pmb{\xi}_{k-1}) \\
   &  \times
  \Big( \prod_{j\in J} \mathbf{1}_{\{ | \xi_j|   \leq  N\} } \Big) \Big( \prod_{\ell \in J^c} \mathbf{1}_{\{ | \xi_\ell|   > N\} }\Big) \prod_{j=1}^{k-1}\exp\Big( -(s-r )  \frac{(v_{j+1} - v_j)v_j }{v_{j+1}}    | \xi_j    |^2 \Big)  \\
   \leq & \sum_{J\subset\{1, \dots, k-1\}}    D_N^{ |J|}  \int_{\dT_{k-1}^1}  d\bv_{k-1}  \int_{\R^{|J^c|d  } }\Big(\prod_{\ell \in J^c} \mu(d \xi_\ell)\Big)\\
  &\times \prod_{\ell \in J^c}  \mathbf{1}_{\{ | \xi_\ell|   > N\}  } \exp  \Big(-(s-r )  \frac{(v_{\ell+1} - v_\ell)v_\ell }{v_{\ell+1}}    | \xi_\ell    |^2 \Big),
  \end{align*}
  where $D_N$ and $C_N$ appearing below  are introduced as in \eqref{CD_N} and $J^c=\{1,\dots,k-1\}\setminus J$.

Suppose  $J^c = \{ \ell_1, \dots , \ell_j\}$ for some $j=0,\dots,k-1$ with $\ell_1 < \dots < \ell_j$. Then, performing the integral with respect to $v_{\ell_1}$, yields
  \begin{align*}
 & \int_{0} ^{v_{\ell_1+1}}dv_{\ell_1}  \exp\Big(  -(s-r )  \frac{(v_{\ell_1+1} - v_{\ell_1})v_{\ell_1} }{v_{\ell_1+1}}    | \xi_{\ell_1}    |^2 \Big)   \notag  \\
  = &v_{\ell_1+1}  \int_0^1 du  \exp\Big(  -(s-r )  [u(1-u) ]   v_{\ell_1+1}   | \xi_{\ell_1}    |^2 \Big)   \\
  = &2v_{\ell_1+1}  \int_0^{1/2}  du  \exp\Big(  -(s-r )  [u(1-u) ]  v_{\ell_1+1}    | \xi_{\ell_1}   |^2 \Big)  \notag  \\
  \leq &2v_{\ell_1+1}  \int_0^{1/2} du  \exp\Big(  -\frac{1}{2}u (s-r )    v_{\ell_1+1}   | \xi_{\ell_1}    |^2 \Big)   =   \frac {4 } { (s-r) | \xi_{\ell_1}    |^2}.  \notag  
  \end{align*}
Applying   this procedure  for the integrals with respect to the variables $v_{\ell_2}$,\dots,$v_{\ell_j}$ successively,    we  obtain
  \[
   \int_{\dT_{k-1}^1}  d\bv_{k-1}  \int_{\R^{jd } }
  \prod_{\ell \in J^c} \mu(d \xi_\ell) \mathbf{1}_{\{ | \xi_\ell|   > N\}  }\exp\Big(  -(s-r )  \frac{(v_{\ell+1} - v_\ell)v_\ell }{v_{\ell+1}}    | \xi_\ell    |^2 \Big) \leq   \Big( \frac{4C_N}{s-r} \Big)^{j}  \frac{1}{  |k-j-1| !}.
  \]
It follows that
  \[
   \mathcal{I} \leq   \sum_{j=0} ^{k-1} \binom{k-1}{j}   \frac{[4(s-r)^{-1}  C_N]^{k-j-1}}{j !}    D_N^{\ell}\leq  2^{k-1}\sum_{j=0} ^{k-1}    \frac{[4(s-r)^{-1}  C_N]^{k-j-1}}{j !}  ,   
  \]
  which implies
  \begin{align} \label{bdd_fkk1}
   \big\| f_{s, y, k}^{(k)}(r,z; \br_{k-1} , \bullet )  \big\|^2_{\kX^{\otimes (k-1)}} 
     &
   \leq \frac{2^{k-1}}{k!} p_{s-r}(y-z)^2  \sum_{\ell=0} ^{k-1}  \frac{[4  C_N]^{k-\ell-1}}{\ell !}    [(s-r)D_N]^{\ell}.
   \end{align}
Combining \eqref{so_far}, \eqref{lem33_HHNT}, \eqref{quad_1}, \eqref{bdd_fkk22} and \eqref{bdd_fkk1},
we can write
  \begin{align}\label{Q_H121}
  \mathcal{Q}^\X_{n,m} \leq & c_1 c_2^n   \sum_{\pmb{i}_m\in [n]_{<}}  \sum_{\ell =0}^{n-i_m}\frac{s_1^{\ell}}{\ell !}D_N^{\ell}C_N^{n-i_m-\ell}  \nonumber\\
   & \times      \prod_{j=1}^m\sum_{\ell=0} ^{i_j-i_{j-1}-1}   \frac{ C_N^{i_j-i_{j-1} - \ell-1}}{\ell !}   [(s_{m-j+2}-s_{m-j+1})D_N]^{\ell}f_{s,y,m}(\bs_m,\by_m)^2. 
  \end{align}
If $C_N=0$ for some $N>0$, then, we have
  \begin{align*}
  \mathcal{Q}^\X_{n,m} \leq & c_1 c_2^n   \sum_{\pmb{i}_m\in [n]_{<}}  \frac{s_1^{n-i_m}}{(n-i_m) !}D_N^{n-i_m}  \\
   & \times      \prod_{j=1}^m   \frac{ 1}{(i_j-i_{j-1}-1) !}   [(s_{m-j+2}-s_{m-j+1})D_N]^{i_j-i_{j-1}-1}f_{s,y,m}(\bs_m,\by_m)^2. \notag 
  \end{align*}
Thus
\begin{align*}
\sum_{n= m}^{\infty} \big[ &(p-1) \Gamma_t\big]^{\frac{n-m}{2}}  \sqrt{\mathcal{Q}^\X_{n,m}}\leq \sum_{ 1\leq  i_1 <  \dots < i_m \leq n<\infty}c_1 c_2^n \big[ (p-1) \Gamma_t\big]^{\frac{n-m}{2}}    \frac{s_1^{\frac{1}{2}(n-i_m)}}{[(n-i_m) !]^{\frac{1}{2}}}D_N^{\frac{1}{2}(n-i_m)}  \\
   & \times      \prod_{j=1}^m   \frac{ 1}{[(i_j-i_{j-1}-1) !]^{\frac{1}{2}}}   [(s_{m-j+2}-s_{m-j+1})D_N]^{\frac{1}{2}(i_j-i_{j-1}-1)}f_{s,y,m}(\bs_m,\by_m)\\
   =&\sum_{i_1=1}^{\infty}\sum_{i_2=i_1+1}^{\infty}\cdots\sum_{ n=i_m+1}^{\infty} c_1 c_2^n \big[ (p-1) \Gamma_t\big]^{\frac{n-m}{2}}   \frac{s_1^{\frac{1}{2}(n-i_m)}}{[(n-i_m) !]^{\frac{1}{2}}}D_N^{\frac{1}{2}(n-i_m)}  \\
   & \times      \prod_{j=1}^m   \frac{ 1}{[(i_j-i_{j-1}-1) !]^{\frac{1}{2}}}   [(s_{m-j+2}-s_{m-j+1})D_N]^{\frac{1}{2}(i_j-i_{j-1}-1)}f_{s,y,m}(\bs_m,\by_m).
\end{align*}
Notice that, by using Lemma \ref{lmm_gmm} (ii) and Corollary \ref{coro_gmm}, we have
\begin{align*}
&\sum_{ n=i_m+1}^{\infty} c_1 c_2^n \big[ (p-1) \Gamma_t\big]^{\frac{n-m}{2}}   \frac{s_1^{\frac{1}{2}(n-i_m)}}{[(n-i_m) !]^{\frac{1}{2}}}D_N^{\frac{1}{2}(n-i_m)}\\
& \qquad  =c_1c_2^{i_m}\big[ (p-1) \Gamma_t\big]^{\frac{i_m-m}{2}}\sum_{ k=1}^{\infty}  c_2^{k}\big[ (p-1) \Gamma_t\big]^{\frac{k}{2}}   \frac{s_1^{\frac{1}{2}(k)}}{[\Gamma(k+1)]^{\frac{1}{2}}}D_N^{\frac{1}{2}k}\\
& \qquad   \leq c_1c_2^{i_m}\big[ (p-1) \Gamma_t\big]^{\frac{i_m-m}{2}}e^{c_3\Gamma_t t}.
\end{align*}
Thus, by iteration, we conclude that
\begin{align}\label{ine_sqt}
\sum_{n= m}^{\infty} \big[ (p-1) \Gamma_t\big]^{\frac{n-m}{2}}  \sqrt{\mathcal{Q}^\X_{n,m}}\leq c_1 e^{c_2\Gamma_t t}f_{s,y,m}(\bs_m,\by_m).
\end{align}
Inequality \eqref{M_D_reg_p} is a consequence of inequalities \eqref{Q_H12} and \eqref{ine_sqt}.

On the other hand, suppose that $C_N>0$ for all $N>0$. From inequality \eqref{Q_H121}, we deduce that
\begin{align*}
 \mathcal{Q}^\X_{n,m} \leq & c_1 c_2^n C_N^{n-m}  \sum_{\pmb{i}_m\in [n]_{<}}  \sum_{\ell =0}^{n-i_m}\frac{s_1^{\ell}}{\ell !}D_N^{\ell}C_N^{-\ell}  \nonumber\\
   & \times      \prod_{j=1}^m\sum_{\ell=0} ^{i_j-i_{j-1}-1}   \frac{ C_N^{- \ell}}{\ell !}   [(s_{m-j+2}-s_{m-j+1})D_N]^{\ell}f_{s,y,m}(\bs_m,\by_m)^2\\
   \leq &c_1 c_2^n \sum_{\pmb{i}_m\in [n]_{<}} C_N^{n-m} e^{\frac{D_N}{C_N}t}f_{s,y,m}(\bs_m,\by_m)^2\leq c_1c_2^nC_N^{n-m}e^{\frac{D_N}{C_N}t}f_{s,y,m}(\bs_m,\by_m)^2.
   \end{align*}
Notice that by definition $\lim_{N\uparrow \infty}C_N=0$.  Therefore, we can find $N$ large enough such that $c_2((p-1)\Gamma_t)^{\frac{1}{2}}C_N<1$, and thus,
\begin{align*}
\sum_{n= m}^{\infty} \big[ (p-1) \Gamma_t\big]^{\frac{n-m}{2}}  \sqrt{\mathcal{Q}^\X_{n,m}}\leq &c_1c_2^me^{\frac{D_N}{C_N}t}f_{s,y,m}(\bs_m,\by_m)\sum_{n= 0}^{\infty} \big[ (p-1) \Gamma_t\big]^{\frac{n}{2}}  c_2^nC_N^{n}\\
=&\frac{c_1c_2^me^{\frac{D_N}{C_N}t}}{1-c_2((p-1)\Gamma_t)^{\frac{1}{2}}C_N}f_{s,y,m}(\bs_m,\by_m).
\end{align*}
 This completes the proof of Theorem \ref{M_D_regular}.\hfill$\square$

\section{Rough case under Hypothesis \ref{h2}}  \label{sec_rough}
In this section, we will deal with the \textbf{rough case}. That is, we consider the parabolic Anderson model \eqref{pam} under Hypothesis \ref{h2} and, as already mentioned in the introduction,  extra effort will be poured into for the spatial roughness. Taking advantage of the Gagliardo representation  \eqref{def_gag} of the inner product on $\mathcal{H}_1$, we apply a modified version of the second-order Gaussian Poincar\'e inequality (see Proposition \ref{propA2}). In order to estimate the quantity $\cA$ in Proposition \ref{propA2}, we need the next proposition about the upper bounds of the Malliavin derivatives and their increments.

\begin{proposition}\label{corodu}
Assume Hypothesis \ref{h2} and  let $u$ be the solution to \eqref{pam}. Given $t\in(0,\infty)$,    for almost  any $0<r<s<t$, $x,y,y',z,z'\in \R$ and for every $p\geq 2$, the following inequalities hold:
\begin{align}\label{cdu}
\|D_{s,y+y'}u(t,x)-D_{s,y}u(t,x)\|_{p}\leq C_1(t) \big(\Phi_{t-s,y'}p_{4(t-s)} \big)(x-y)
\end{align}
and
\begin{align}\label{cddu}
\|D^2_{r,z+z',s,y+y'}u(t,x)&-D^2_{r,z+z',s,y}u(t,x)-D^2_{r,z,s,y+y'}u(t,x)+D^2_{r,z,s,y}u(t,x)\|_{p}\nonumber\\
\leq & C_1(t)\Lambda_{r,z',s,y'} (p_{4(t-s)},p_{4(s-r)})(x-y,y-z),
\end{align}
where we fix $\Phi=\Phi^{H_0-\frac{1}{4}}$, defined as in Lemma \ref{tl1}, $\Lambda$ is defined as in \eqref{def_lmd}, and $C_1(t)=c_1\exp(c_2 t^{\frac{2H_0+H_1-1}{2H_1}})$ for all $t>0$ with some constants $c_1$ and $c_2$ depending on $H_0$ and $H_1$.
\end{proposition}
The proof of Proposition \ref{corodu} is based on the following lemmas. We firstly show how these lemmas imply Proposition \ref{corodu}. The proofs of Lemmas \ref{produ} and \ref{prod2u},  which  heavily rely on the Wiener chaos expansion, are postponed to Section \ref{sec5}.
\begin{lemma}\label{produ}
Assume Hypothesis \ref{h2} and let $u$ be the solution to \eqref{pam}. Then, for almost every $(s, x,y,y')  \in(0,t) \times \R^3$ and for any $p\geq 2$,  the following inequalities hold:
\begin{align}\label{du}
\|D_{s,y}u(t,x)\|_{p}\leq C_1(t)p_{t-s}(x-y)
\end{align}
and
\begin{align}\label{ddu}
\|D_{s,y+y'}u(t,x)-D_{s,y}u(t,x)\|_{p}\leq C_1(t) \Big(  |\Delta_{t-s}(y-x,y')|+p_{t-s}(x-y)N_{t-s}(y')\Big),
\end{align}
where $\Delta_t$ and $N_t$ are defined as in \eqref{delta} and \eqref{n}, respectively,  and $C_1(t)$ is the same as in Proposition \ref{corodu}.
\end{lemma}

\begin{lemma}\label{prod2u}
Assume Hypothesis \ref{h2}. Let $u$ be the solution to \eqref{pam}. Then, for almost any $0<r<s<t$, $x,y,y',z,z'\in \R$ and  for every $p\geq 2$,   the following inequality holds:
\begin{align}\label{dd2u}
&\|D^2_{r,z+z',s,y+y'}u(t,x)-D^2_{r,z+z',s,y}u(t,x)-D^2_{r,z,s,y+y'}u(t,x)+D^2_{r,z,s,y}u(t,x)\|_{p}   \notag  \\
\leq &C_1(t)\bigg\{p_{t-s}(x-y)N_r(z') \Big[  |\Delta_{s-r}(y-z,y')|+p_{s-r}(y-z)N_{s-r}(y')\Big]   \notag\\
&+p_{t-s}(x-y) \Big[  |R_{s-r}(y-z,y',z')|+|\Delta_{s-r}(y-z,y')| N_{s-r}(z')    \notag  \\
&\qquad\qquad+|\Delta_{s-r} (z-y,z')|   N_{s-r}(y')+p_{s-r}(y-z) N_{s-r}(y')  N_{s-r}(z')\Big]   \notag  \\
&+p_{s-r}(y+y'-z) N_r(z')  \Big[  |\Delta_{t-s}(y-x,y')|+p_{t-s}(x-y) N_{t-s}( y') \Big]  \notag \\
&+\Big[    |\Delta_{s-r}( z-y-y',z')|+p_{s-r}(y+y'-z) N_{s-r}(z') \Big]   \notag   \\
&\qquad\times\Big[  |\Delta_{t-s}( y-x,y')|+p_{t-s}(x-y)  N_{t-s}( y')\Big]\bigg\},
\end{align}
where $\Delta_t$, $R_t$ and $N_t$ are defined as in \eqref{delta} - \eqref{n}, respectively, and $C_1$ is the same as in  Proposition \ref{corodu}.
\end{lemma}

\begin{proof}[Proof of Proposition \ref{corodu}] Let us first recall  from  \eqref{n} that
\[
N_t(x) = t^{\frac{1}{8} - \frac{1}{2}H_0}|x|^{H_0-\frac{1}{4}}\1_{\{|x|\leq \sqrt{t}\}}+\1_{\{|x|>\sqrt{t}\}}.
\]
Then   applying    Lemma \ref{tl1} with  $\beta=H_0-\frac{1}{4}$ yields immediately    that
\begin{align}
\Phi_{t,x'}g(x) \equiv \Phi_{t,x'}^{H_0 - \frac14}g(x)    =\theta_{x'}g(x)\1_{\{  |x'|>\sqrt{t} \} }+N_t(x')g(x) \geq N_{t}(x')g(x) \label{use0_1}
\end{align}
for any nonnegative  function $g\in\cM(\R)$. Now combining inequalities \eqref{ddu} and \eqref{delta0}, we get \eqref{cdu} immediately. 

In the next step, we prove inequality \eqref{cddu} that contains  more   terms. This is because the bound in \eqref{dd2u} is a sum of 4 terms, which we denote by $R_1, R_2, R_3$ and  $R_4$.

\medskip
Firstly, we estimate $R_1$ and $R_2$ by using inequality \eqref{use0_1} and Lemma \ref{tl1} as follows,
\begin{align*}
R_1:=&p_{t-s}(x-y)N_r(z') \big(|\Delta_{s-r}(y-z,y')|+p_{s-r}(y-z)N_{s-r}(y')\big)   \\
\leq &c \,  p_{4(t-s)}(x-y)N_{r}(z')   \big(   \Phi_{s-r,y'} p_{4s-4r}(y-z) +p_{4s-4r}(y-z)N_{s-r}(y')\big)   \\
\leq & c \,  p_{4(t-s)}(x-y)N_{r}(z')  \big(  \Phi_{s-r,y'} p_{4s-4r}\big)(y-z),
\end{align*}
and  taking  into account Remark \ref{rem_0_1} and the inequality $p_t(x) \leq 2 p_{4t}(x)$, we can write
\begin{align*}
R_2:&= p_{t-s}(x-y) \big[  |R_{s-r}(y-z,y',z')|+|\Delta_{s-r}(y-z,y')| N_{s-r}(z')       \\
&\qquad\qquad+|\Delta_{s-r} (z-y,z')|   N_{s-r}(y')+p_{s-r}(y-z) N_{s-r}(y')  N_{s-r}(z')\big] \\
\leq &c\, p_{t-s}(x-y) \big[ (\Phi_{s-r,y'}\Phi_{s-r,-z'}p_{4(s-r)})(y-z)+(\Phi_{s-r,y'}p_{4(s-r)})(y-z) N_{s-r}(z') \\
&\qquad\qquad \qquad +(\Phi_{s-r,z'}p_{4(s-r)})(z-y)N_{s-r}(y')+p_{4s-4r}(y-z)N_{s-r}(y')N_{s-r}(z')  \big]  \\
\leq &c\, p_{t-s}(x-y) \big[  (\Phi_{s-r,y'}\Phi_{s-r,-z'}p_{4(s-r)})(y-z)+(\Phi_{s-r,-z'}\Phi_{s-r,y'}p_{4(s-r)})(y-z)  \\
&\qquad\qquad\qquad +(\Phi_{s-r,-y'}\Phi_{s-r,z'}p_{4(s-r)})(z-y)+ (\Phi_{s-r,y'}\Phi_{s-r,-z'}p_{4s-4r})(y-z)\big]  \\
\leq &c\, p_{4(t-s)}(x-y) \big(\Phi_{s-r,y'}\Phi_{s-r,-z'}p_{4(s-r)}\big)(y-z).
\end{align*}
 Following a similar argument, we also get
\begin{align*}
R_3:=&p_{s-r}(y+y'-z) N_r(z')\big[  |\Delta_{t-s}(y-x,y')| + p_{t-s}(x-y) N_{t-s}( y')\big]  \\
\leq &c\, p_{4(s-r)}(y+y'-z) N_r(z') \big( \Phi_{t-s,-y'}p_{4(t-s)} \big)(x-y)   \\
=&c\, (\theta_{y'}p_{4(s-r)})(y-z) N_r(z') \big( \Phi_{t-s,-y'}p_{4(t-s)} \big)(x-y)
\end{align*}
and
\begin{align*}   
R_4:=&\big[ |\Delta_{s-r}( z-y-y',z')|+p_{s-r}(y+y'-z) N_{s-r}(z')\big]  \\
 &\qquad \times \big[ |\Delta_{t-s}( y-x,y')|+p_{t-s}(x-y)N_{t-s}( y')\big]  \\
\leq &c\, (\Phi_{s-r, -z'}p_{4s-4r})(z-y-y')(\Phi_{t-s,-y'}p_{4(t-s)})(x-y) \\
=&c\,  (\theta_{y'}\Phi_{s-r,-z'}p_{4s-4r})(y-z) \big( \Phi_{t-s,-y'}p_{4(t-s)} \big)(x-y).
\end{align*}
Using  the above estimates, inequality  \eqref{cddu}  follows immediately.
\end{proof}

In what follows, we first give  the remaining proof of  \eqref{main_TV} in Section \ref{sec41}.  Later, in Section \ref{sec5} provides proofs of several auxiliary results.

\subsection{Proof of  \texorpdfstring{\eqref{main_TV}}{} in the rough case}\label{sec41}

According to Proposition \ref{propA2}, we need to estimate the quantity $\cA$ defined as in \eqref{def_ca2} with $F=F_R$ given as in \eqref{def_FRT}.  Our goal is to show $\mathcal{A} \lesssim R$ for large $R$. Indeed, we already know from Theorem \ref{thm_quality} that $\sigma^2_R(t) \sim R$, then the desired bound \eqref{main_TV} (in the \textbf{rough case}) follows.

\medskip

In what follows,  we only provide a detailed proof assuming $\gamma_0(s) = |s|^{2H_0-2}$ for $H_0\in(1/2,1)$, while the other case   ($\gamma_0 = \delta_0$) can be dealt with in the same way.  We first write
\begin{align*}
\cA \leq \int_{[0,t]^6} dsds'  drdr' d\theta  d\theta'    |s-s'|^{2H_0-2}  |r-r'|^{2H_0-2}  |\theta - \theta'|^{2H_0-2}  \cA_0,
\end{align*}
where, using a changing of variables in space,
\begin{align*}
 \cA_0& := \int_{\R^6}    dy dy'  dz dz'  dw dw'     \int_{[-R,R]^4}dx_1dx_2dx_3dx_4   |  y'|^{2H_1-2} |  z'|^{2H_1-2}  |  w'|^{2H_1-2}  \\
& \times \big\| D_{r', z+z'}u(t,x_1) - D_{r', z}u(t,x_1)\big\|_4 \big\| D_{\theta', w+w'} u(t,x_2) - D_{\theta', w} u(t,x_2)\big\|_4  \\
& \times \big\| D_{s,y+y'}D_{r,z+z'} u(t,x_3) - D_{s,y+y'}D_{r,z} u(t,x_3)    -D_{s,y}D_{r,z+z'} u(t,x_3) + D_{s,y}D_{r,z} u(t,x_3)     \big\|_4 \\
&   \times  \big\|  D_{s',y+y'}D_{\theta,w+w'} u(t,x_4) - D_{s',y+y'}D_{\theta, w} u(t,x_4)    -D_{s',y}D_{\theta,w+w'} u(t,x_4) \\
 &\qquad+ D_{s',y}D_{\theta,w}u(t,x_4)\big\|_4.
\end{align*}

Next,  we will estimate  $\cA_0$. Suppose $0<r<s<t$, and $0<\theta< s'<t$. We deduce from Proposition \ref{corodu} that 
\begin{align*}
\cA_0 \lesssim & \int_{[-R,R]^4}dx_1dx_2dx_3dx_4\int_{\R^6} dy dy' dz dz'  dw dw'  | y'|^{2H_1-2} |z'|^{2H_1-2}| w'|^{2H_1-2}\\
&\times  \big(\Phi_{t-r',z'}p_{4(t-r')}\big) (x_1-z) \big( \Phi_{t-\theta',w'}p_{4(t-\theta')} \big)(x_2-w) \\
&\times \Lambda_{r,z',s,y'} (p_{4(t-s)},p_{4(s-r)})(x_3-y,y-z)\Lambda_{\theta,w',s',y'} (p_{4(t-s')},p_{4(s'-\theta)})(x_4-y,y-w).
\end{align*}
Now we first  integrate  out $x_3$, $x_4$ and $x_2$ one by one (see Remark \ref{rem_Lambda}):
\begin{align*}
\cA_0 \lesssim & \int_{-R}^Rdx_1\int_{\R^6} dy dy' dz dz'   dw dw'  | y'|^{2H_1-2} |z'|^{2H_1-2}| w'|^{2H_1-2}\\
&\times  \big(\Phi_{t-r',z'}p_{4(t-r')}\big) (x_1-z) \big( \Phi_{t-\theta',w'}  \1_\R \big)(0) \\
&\times \Lambda_{r,z',s,y'} ( \1_\R,p_{4(s-r)})(0,y-z)\Lambda_{\theta,w',s',y'} ( \1_\R,p_{4(s'-\theta)})(0,y-w),
\end{align*}
and then integrate  out $w$, $y$ and $z$ to get
\begin{align*}
\cA_0 \lesssim & \int_{-R}^Rdx_1 \int_{\R^3} dy' dz'    dw'  | y'|^{2H_1-2} |z'|^{2H_1-2}| w'|^{2H_1-2}\\
&\times  \big(\Phi_{t-r',z'} \1_\R\big) (0) \big( \Phi_{t-\theta',w'}  \1_\R \big)(0)   \Lambda_{r,z',s,y'} ( \1_\R, \1_\R)(0, 0)\Lambda_{\theta,w',s',y'} ( \1_\R, \1_\R)(0, 0).
\end{align*}
Applying Cauchy-Schwarz inequality, we can further deduce that
\begin{align*}
\cA_0\lesssim &\,\, R \int_{\R} dy' | y'|^{2H_1-2}  \Big( \int_{\R^2}  dz'    dw'   |z'|^{2H_1-2}| w'|^{2H_1-2}  \big\vert \big(\Phi_{t-r',z'} \1_\R\big) (0) \big( \Phi_{t-\theta',w'}  \1_\R \big)(0) \big\vert^2 \Big)^{1/2}  \\
&\times  \Big( \int_{\R^2} dz'    dw'   |z'|^{2H_1-2}| w'|^{2H_1-2}  \big\vert \Lambda_{r,z',s,y'} ( \1_\R, \1_\R)(0, 0)\Lambda_{\theta,w',s',y'} ( \1_\R, \1_\R)(0, 0) \big\vert^2 \Big)^{1/2}\\
\lesssim &\,\,R  \Big( \int_{\R^2}  dz'    dw'   |z'|^{2H_1-2}| w'|^{2H_1-2}  \big\vert \big(\Phi_{t-r',z'} \1_\R\big) (0) \big( \Phi_{t-\theta',w'}  \1_\R \big)(0) \big\vert^2 \Big)^{1/2}  \\
&\times \Big( \int_{\R^2} dy'  dz'    | y'|^{2H_1-2}       |z'|^{2H_1-2}    \big\vert \Lambda_{r,z',s,y'} ( \1_\R, \1_\R)(0, 0)  \big\vert^2  \Big)^{1/2}  \\
&  \times  \Big( \int_{\R^2} dy'  dw'| y'|^{2H_1-2}     w'|^{2H_1-2}  \big\vert  \Lambda_{\theta,w',s',y'} ( \1_\R, \1_\R)(0, 0) \big\vert^2 \Big)^{1/2}.
\end{align*}
Due to the fact that $2H_1+2H_0-\frac{5}{2}>- 1$ and $2H_1 <1$, we have 
\begin{align}
&\quad \int_{\R}dz' |z'|^{2H_1-2}\big|(\Phi_{t-r',z'}\1_{\R})(0)\big|^2  \notag  \\
& = 8\int_{\sqrt{t-r'}}^{\infty}dz' |z'|^{2H_1-2}+ 2 (t-r')^{\frac{1}{4}-H_0}\int_{0}^{\sqrt{t-r'}}|z'|^{2H_1+2H_0-\frac{5}{2}}dz' \lesssim (t-r')^{H_1-\frac{1}{2}}.\label{ineq_PHI}
\end{align}
We can also deduce the next inequality by definition of $\Lambda$ and Remark \ref{rem_0_1},
 \begin{align*}
 &\quad  \big\vert \Lambda_{r,z',s,y'} (\1_{\R},\1_{\R})(0,0) \big\vert^2 \\
 &\lesssim \big\vert  N_{s-r}(y')N_r(z')+N_{s-r}(y')N_{s-r}(z')  +N_{t-s}(y')N_r(z')+N_{t-s}(y')N_{s-r}(z')\big|^2 \\
 &\lesssim N_{s-r}(y')^2 N_r(z')^2+N_{s-r}(y')^2N_{s-r}(z')^2  +N_{t-s}(y')^2N_r(z')^2+N_{t-s}(y')^2N_{s-r}(z')^2,
 \end{align*}
 which, together with \eqref{ineq_nn}, implies 
\begin{align} 
&\int_{\R^2} dy' dz' |y'|^{2H_1-2}|z'|^{2H_1-2} |\Lambda_{r,z',s,y'} (\1_{\R},\1_{\R})(0,0)|^2 \notag  \\
\lesssim&   (s-r)^{H_1-\frac{1}{2}}r^{H_1-\frac{1}{2}}+(s-r)^{2H_1-1}+(t-s)^{H_1-\frac{1}{2}}r^{H_1-\frac{1}{2}}+(t-s)^{H_1-\frac{1}{2}}(s-r)^{H_1-\frac{1}{2}}.\label{ineq_LAM}
\end{align}
 As a consequence of \eqref{ineq_PHI} and \eqref{ineq_LAM}, we get $\cA_0 \lesssim \cB_0 R $ with
\begin{align*}
& \cB_0:=   (t-r')^{\frac{H_1}{2} - \frac{1}{4}}(t-\theta')^{\frac{H_1}{2} - \frac{1}{4}}  \\
&\times \big[(s-r)^{\frac{H_1}{2} - \frac{1}{4}}r^{\frac{H_1}{2} - \frac{1}{4}}+(s-r)^{H_1-\frac{1}{2}}+(t-s)^{\frac{H_1}{2} - \frac{1}{4}}r^{\frac{H_1}{2} - \frac{1}{4}}+(t-s)^{\frac{H_1}{2} - \frac{1}{4}}(s-r)^{\frac{H_1}{2} - \frac{1}{4}}\big]\\
&\times \big[(s'-\theta)^{\frac{H_1}{2} - \frac{1}{4}}\theta^{\frac{H_1}{2} - \frac{1}{4}}+(s'-\theta)^{H_1-\frac{1}{2}}+(t-s')^{\frac{H_1}{2} - \frac{1}{4}}\theta^{\frac{H_1}{2} - \frac{1}{4}}+(t-s')^{\frac{H_1}{2} - \frac{1}{4}}(s'-\theta)^{\frac{H_1}{2} - \frac{1}{4}}\big].
\end{align*}
Notice that with $\gamma_0(s) = |s|^{2H_0-2}$ for  $H_0\in(1/2, 1)$. This allows us to apply the embedding inequality \eqref{embd} and get
\begin{align*}
& \quad \int_{\substack{0<r<s<t \\ 0<\theta<s'<t}  }drds d\theta ds'\int_{[0,t]^2}dr'd\theta'  \gamma_0(s-s') \gamma_0(r-r') \gamma_0(\theta - \theta' )   \cB_0 \\
& \lesssim \Big\{ \int_{0<r<s<t} drds \int_0^t d\theta' (t-\theta')^{\frac{H_1}{2H_0} - \frac{1}{4H_0}} \big((s-r)^{\frac{H_1}{2} - \frac{1}{4}}r^{\frac{H_1}{2} - \frac{1}{4}}+(s-r)^{H_1-\frac{1}{2}} \\
 &\qquad\qquad\qquad +(t-s)^{\frac{H_1}{2} - \frac{1}{4}}r^{\frac{H_1}{2} - \frac{1}{4}}+(t-s)^{\frac{H_1}{2} - \frac{1}{4}}(s-r)^{\frac{H_1}{2} - \frac{1}{4}}\big)^{1/H_0} \Big\}^{2H_0} < +\infty.
\end{align*}
Therefore, 
\begin{align*}
 \int_{\substack{0<r<s<t \\ 0<\theta<s'<t } }drds d\theta ds'\int_{[0,t]^2}dr'd\theta'    \gamma_0(s-s') \gamma_0(r-r') \gamma_0(\theta - \theta' )      \cA_0    \lesssim R.
\end{align*}
For the case $\gamma_0=\delta_0$, the expression for $\cB_0$ reduces to 
\begin{align*}
& \cB_0=   (t-r)^{\frac{H_1}{2} - \frac{1}{4}}(t-\theta)^{\frac{H_1}{2} - \frac{1}{4}}\\
&\times \big[(s-r)^{\frac{H_1}{2} - \frac{1}{4}}r^{\frac{H_1}{2} - \frac{1}{4}}+(s-r)^{H_1-\frac{1}{2}}+(t-s)^{\frac{H_1}{2} - \frac{1}{4}}r^{\frac{H_1}{2} - \frac{1}{4}}+(t-s)^{\frac{H_1}{2} - \frac{1}{4}}(s-r)^{\frac{H_1}{2} - \frac{1}{4}}\big]\\
&\times \big[(s-\theta)^{\frac{H_1}{2} - \frac{1}{4}}\theta^{\frac{H_1}{2} - \frac{1}{4}}+(s-\theta)^{H_1-\frac{1}{2}}+(t-s)^{\frac{H_1}{2} - \frac{1}{4}}\theta^{\frac{H_1}{2} - \frac{1}{4}}+(t-s)^{\frac{H_1}{2} - \frac{1}{4}}(s-\theta)^{\frac{H_1}{2} - \frac{1}{4}}\big].
\end{align*}
and for the same reason as above, 
\[
 \int_{\substack{0<r<s<t \\ 0<\theta<s<t } }drds d\theta    \cA_0  \lesssim R \int_{\substack{0<r<s<t \\ 0<\theta<s<t } }drds d\theta    \cB_0  \lesssim R.
\]
The remaining three cases
\[
\begin{cases}
\text{$0<s<r<t$ and $0<\theta< s'<t$}\\
\text{$0<r<s<t$ and $0<s'<\theta<t$}\\
    \text{$0<s<r<t$ and $0<s'<\theta<t$}
\end{cases}
\]
can be   estimated in an almost same  way and finally we can  get the same upper bounds. That is, we obtain the desired bound  $\mathcal{A}\lesssim R$ and hence conclude the proof of \eqref{main_TV} under Hypothesis \ref{h2}. \hfill $\square$

\subsection{Proof of some auxiliary results}\label{sec5}
In this subsection, we introduce some auxiliary results and provide the proof of Lemma \ref{produ} and \ref{prod2u} in Section \ref{ss.pfp1}.

\subsubsection{Estimates for fixed Wiener chaoses}\label{ss.estgk}
Fix $0<s<t<\infty$ and $x,y\in \R$. For any $n=1,2,\dots$, let $g_{s,t,x,n}^1$ and $g_{s,y,t,x,n}^2$ be functions on $[s,t]^n\times \R^n$ given by
\begin{align}\label{gn1}
g_{s, t,x,n}^1(\bs_n,\by_n)=n!f_{t,x,n}(\bs_n,\by_n)\1_{[s,t]^n}(\bs_n),
\end{align}
and
\begin{align}\label{gn2}
g_{s,y,t,x,n}^2(\bs_n,\by_n)=p_{s_{\sigma(1)}-s}(y_{\sigma(1)}-y)g_{s, t,x,n}^1(\bs_n,\by_n),
\end{align}
where $f_{t,x,n}$ is defined as in \eqref{cos2} and $\sigma$, because of the indicator function $\1_{[s,t]^n}$, is now a permutation on $\{1,\dots, n\}$ such that $s<s_{\sigma(1)}<\dots <s_{\sigma(n)}<t$. 

In Section \ref{ss.pfp1} below, we will see that $g^1$ and $g^2$ are closely related to the chaos coefficients of the Mallivain derivatives of $u$. In fact, the next lemmas, which give some estimates for $g_n^1$ and $g_n^2$, are essential to the proofs of Lemmas \ref{produ} and \ref{prod2u}.
\begin{lemma}\label{lmgn1}
Let $0<s<t<\infty$ and let $x\in \R$. Fix a positive integer $n$, and let $g_n^1$ be given as in \eqref{gn1}. Then, the following equalities hold.
\begin{align}\label{gn11}
\Big(\int_{[s,t]^n}d\bs_n \|g^{1}_{s,t,x,n}(\bs_n,\bullet)\|_{\cH_1^{\otimes n}}^{\frac{1}{H_0}}\Big)^{2H_0}\leq C_2(n,t-s) 
\end{align}
and
\begin{align}\label{gn12}
&\Big(\int_{[s,t]^n}d\bs_n\|g_{s, t,x+x',n}^1(\bs_n,\bullet)-g_{s, t,x,n}^1(\bs_n,\bullet)\|_{\cH_1^{\otimes n}}^{\frac{1}{H_0}}\Big)^{2H_0}\leq  C_2(n,t-s)N_{t-s}( x'),
\end{align}
where $N_t$ is defined as in \eqref{n} and $ C_2(n,t)=c_1c_2^n\Gamma((1-H_1)n+1)t^{(2H_0+H_1-1)n}$ for all positive integers $n$ and real numbers  $t>0$ with some constants $c_1, c_2$ depending on $H_0$ and $H_1$.
\end{lemma}
\begin{proof}
Fix $s<s_1<s_2<\dots<s_n<t$. Denote by $\hg_n^1$ the Fourier transformation of $g_n^1$ with respect to the spatial arguments. Following  the idea in  \cite[Theorem 3.4]{abel-18-hu-huang-le-nualart-tindel}, we can show that
\begin{align}\label{gn111}
\|g^{1}_{s, t,x,n}(\bs_n,\bullet)&\|_{\cH_1^{\otimes n}}^2=c_{H_1}^n\int_{\R^n}d\bxi_n \prod_{j=1}^n|\xi_j|^{1-2H_1} |\hg^{1}_{s,t,x,n}(\bs_n,\bxi_n)|^2\nonumber\\
\leq &c_{H_1}^n\sum_{\balpha_k\in D_n}\int_{\R^n}d\pmb{\eta}_n |\eta_1|^{1-2H_1}\prod_{j=1}^n|\eta_j|^{\alpha_j}\times \prod_{j=1}^{n-1}e^{-(s_{j+1}-s_j)|\eta_j|^2}\times e^{-(t-s_n)|\eta_n|^2}\nonumber\\
\leq & c_1c_2^n\sum_{\balpha_n\in D_n}(s_2-s_1)^{-\frac{2-2H_1+\alpha_1}{2}}\times \prod_{j=2}^{n}(s_{j+1}-s_j)^{-\frac{1+\alpha_j}{2}},
\end{align}
where $D_n$ is a collection of multi-indexes $\balpha_k=(\alpha_1,\dots,\alpha_n)$ with
\begin{align}\label{dk1}
\alpha_1,\alpha_n\in\{0,1-2H_1\},\ \alpha_i\in \{0,1-2H_1,2(1-2H_1)\},\ \forall i=2,\dots, n-1,
\end{align}
and
\begin{align}\label{dk2}
\sum_{i=1}^n\alpha_i=(n-1)(1-2H_1).
\end{align}
Applying \cite[Lemma 4.5]{HHNT15}, and the fact that
\[
\Big(\sum_{i=1}^m x_i\Big)^{\gamma}\leq m^{\gamma}\sum_{i=1}^mx_i^{\gamma}
\]
for all $m=1,2,\dots$, and $x_1,\dots,x_m,\gamma>0$, we get from \eqref{gn111} that
\begin{align}\label{gk2}
\int_{[s,t]^n}d\bs_n \|g^{1}_{s,t,x,n}(\bs_n,\bullet)\|_{\cH_1^{\otimes n}}^{\frac{1}{H_0}}=&n!\int_{\dT^{s,t}_n}d\bs_n\|g^{1}_{s,t,x,n}(\bs_n,\bullet)\|_{\cH_1^{\otimes n}}^{\frac{1}{H_0}}\nonumber\\
\leq &c_1c_2^nn!(t-s)^{\frac{2H_0+H_1-1}{2H_0}n}\Gamma\Big(\frac{2H_0+H_1-1}{2H_0}n+1\Big)^{-1}.
\end{align}
Inequality \eqref{gn11} is thus a consequence of inequality \eqref{gk2} and Corollary \ref{coro_gmm}. 

The proof of \eqref{gn12} is quite similar. Fix $s<s_1<\dots<s_n<t$. By the Fourier transformation, we can write
\begin{align*}
J_n:=&\|g^1_{s,t,x+x',n}(\bs_n,\bullet)-g^1_{s, t,x,n}(\bs_n,\bullet)\|_{\cH_1^{\otimes n}}^2\\
=&c_{H_1}^n\int_{\R^n}d\bxi_n \prod_{j=1}^n|\xi_j|^{1-2H_1} |\hg_{s,t,x+x',n}^{1}(\bs_{n}, \bxi_{n})-\hg_{s,t,x,n}^{1}(\bs_{n}, \bxi_{n})|^2\nonumber\\
\leq &c_{H_1}^n\sum_{\balpha\in D_n}\int_{\R^n}d\pmb{\eta}_n |\eta_1|^{1-2H_1}\prod_{j=1}^n|\eta_j|^{\alpha_j}\times \prod_{j=1}^n e^{-(s_{j+1}-s_j)|\eta_j|^2}|e^{-i(x+x')\eta_n}-e^{-ix\eta_n}|^2,
\end{align*}
where $D_n$ is a set of multi-indexes defined as in \eqref{dk1} and \eqref{dk2}. Using the elementary calculus, we can show that for all $x\in \R$,
\begin{align}\label{eixz}
|e^{-ix}-1|\leq |x|\1_{\{|x|\leq 1\}}+2\times \1_{\{|x|> 1\}},
\end{align}
and thus,
\begin{align*}
J_n\leq &c_1c_2^n\sum_{\balpha\in D_n}\Big[|x'|^2\int_{\R^n}d\pmb{\eta}_n |\eta_1|^{1-2H_1}\prod_{j=1}^{n}|\eta_j|^{\alpha_j}|\eta_n|^{2} \prod_{j=1}^ne^{-(s_{j+1}-s_j)|\eta_j|^2}\1_{\{|x'\eta_n|\leq 1\}}\\
&+\int_{\R^n}d\pmb{\eta}_n |\eta_1|^{1-2H_1}\prod_{j=1}^{n}|\eta_j|^{\alpha_j}\times \prod_{j=1}^ne^{-(s_{j+1}-s_j)|\eta_j|^2}\1_{\{|x'\eta_n|> 1\}}\Big]\\
\leq &c_1c_2^n\sum_{\balpha_n\in D_n}(s_2-s_1)^{-\frac{2-2H_1+\alpha_1}{2}}\times \prod_{j=2}^{n-1}(s_{j+1}-s_j)^{-\frac{1+\alpha_j}{2}}\\
&\times \Big[|x'|^2\int_0^{|x'|^{-1}}d\eta_n|\eta_n|^{2+\alpha_n}e^{-(t-s_n)\eta_n^2}+\int_{|x'|^{-1}}^{\infty}d\eta_n|\eta_n|^{\alpha_n}e^{-(t-s_n)\eta_n^2}\Big].
\end{align*}
Preforming the changing of variable $\sqrt{t-s_n}\eta_n=\eta$, we can write
\begin{align*}
\int_0^{|x'|^{-1}}&d\eta_n|\eta_n|^{2+\alpha_n}e^{-(t-s_n)\eta_n^2}=(t-s_n)^{-\frac{3}{2} - \frac{1}{2}\alpha_n}\int_{0}^{\frac{\sqrt{t-s_n}}{|x'|}}d\eta|\eta|^{2+\alpha_n}e^{-\eta^2}\\
\leq &(t-s_n)^{-\frac{3}{2} - \frac{1}{2}\alpha_n}\Big(\int_{0}^{\infty}d\eta|\eta|^{2+\alpha_n}e^{-\eta^2}\1_{\{|x'|\leq \sqrt{t-s_n}\}}\\
&\quad+\int_{0}^{\frac{\sqrt{t-s_n}}{|x'|}}d\eta|\eta|^{2+\alpha_n}\1_{\{|x'|> \sqrt{t-s_n}\}}\Big)\\
\leq &c_1(t-s_n)^{-\frac{3}{2} - \frac{1}{2}\alpha_n}\big(\1_{\{|x'|\leq \sqrt{t-s_n}\}}+(t-s_n)^{\frac{3+\alpha_n}{2}}|x'|^{-3-\alpha_n}\1_{\{|x'|> \sqrt{t-s_n}\}}\big)\\
\leq &c_1\big((t-s_n)^{-\frac{1}{4} - \frac{1}{2}\alpha_n-H_0}|x'|^{-\frac{5}{2}+2H_0}\1_{\{|x'|\leq \sqrt{t-s_n}\}}+(t-s_n)^{-\frac{1+\alpha_n}{2}}|x'|^{-2}\1_{\{|x'|> \sqrt{t-s_n}\}}\big).
\end{align*}
Similarly, we can also show that
\begin{align*}
\int_{|x'|^{-1}}^{\infty}&d\eta_n|\eta_n|^{\alpha_n}e^{-(t-s_n)\eta_n^2}=(t-s_n)^{-\frac{1}{2} - \frac{1}{2}\alpha_n}\int_{\frac{\sqrt{t-s_n}}{|x'|}}^{\infty}d\eta_n|\eta_n|^{\alpha_n}e^{-\eta_n^2}\\
\leq &c_1(t-s_n)^{-\frac{1}{2} - \frac{1}{2}\alpha_n}\bigg(\Big(\frac{|x'|}{t-s_n}\Big)^{2}\1_{\{|x'|\leq \sqrt{t-s_n}\}}+\1_{\{|x'|> \sqrt{t-s_n}\}}\bigg)\\
\leq &c_1\big((t-s_n)^{-\frac{1}{4} - \frac{1}{2}\alpha_n-H_0}|x'|^{-\frac{5}{2}+2H_0}\1_{\{|x'|\leq \sqrt{t-s_n}\}}+(t-s_n)^{-\frac{1+\alpha_n}{2}}|x'|^{-2}\1_{\{|x'|> \sqrt{t-s_n}\}}\big).
\end{align*}
Therefore,
\begin{align*}
J_n\leq & c_1c_2^n\sum_{\balpha_n\in D_n}(s_2-s_1)^{-\frac{2-2H_1+\alpha_1}{2}}\times \prod_{j=2}^{n-1}(s_{j+1}-s_j)^{-\frac{1+\alpha_j}{2}}\\
&\times \big((t-s_n)^{-\frac{1}{4} - \frac{1}{2}\alpha_n-H_0}|x'|^{-\frac{1}{2}+2H_0}\1_{\{|x'|\leq \sqrt{t-s_n}\}}+(t-s_n)^{-\frac{1+\alpha_n}{2}}\1_{\{|x'|> \sqrt{t-s_n}\}}\big).
\end{align*}
By using \cite[Lemma 4.5]{HHNT15} again, 
we get the following inequality,
\begin{align}\label{jnmin}
\Big(n!\int_{\dT_n^{s,t}}d\bs_n J_n^{\frac{1}{2H_0}}\Big)^{2H_0}\leq &c_1c_2^n\Big[n!\Gamma\Big(\frac{2H_0+H_1-1}{2H_0}n+1\Big)^{-1}(t-s)^{\frac{2H_0+H_1-1}{2H_0}n}\nonumber\\
&\times \big(|x'|^{\frac{-1+4H_0}{4H_0}}(t-s)^{\frac{1-4H_0}{8H_0}}\1_{\{|x'|\leq \sqrt{t-s}\}}+\1_{\{|x'|>\sqrt{t-s}\}}\big)\Big]^{2H_0}.
\end{align}
Hence, inequality \eqref{gn12} follows from inequality \eqref{jnmin} and Corollary \ref{coro_gmm}. The proof of Lemma \ref{lmgn1} is competed.
\end{proof}

\begin{lemma}\label{lmgn2}
Let $0<s<t<\infty$ and let $x,y\in \R$. For any $n=1,2,\dots$, let $g_n$ be given in \eqref{gn2}. Then, the following equalities hold.
\begin{align}\label{gn21}
&\Big(\int_{[s,t]^n}d\bs_{n}\|g_{s,y,t,x,n}^2(\bs_{n}, \bullet)\|_{\cH_1^{\otimes{n}}}^{\frac{1}{H_0}}\Big)^{2H_0}\leq  C_2(n, t-s)p_{t-s}(x-y)^2,
\end{align}
\begin{align}\label{gn22}
&\Big(\int_{[s,t]^n}d\bs_n\|g_{s,y+y', t,x,n}^2(\bs_n,\bullet)-g_{s, y, t,x,n}^2(\bs_n,\bullet)\|_{\cH_1^{\otimes n}}^{\frac{1}{H_0}}\Big)^{2H_0}\nonumber\\
\leq  &  C_2(n,t-s) \big(\big|\Delta_{t-s}( y-x, y')\big|+p_{t-s}(x-y) N_{t-s}(y')\big)^2,
\end{align}
\begin{align}\label{gn23}
&\Big(\int_{[s,t]^n}d\bs_n\|g_{s,y, t,x+x',n}^2(\bs_n,\bullet)-g_{s,y, t,x,n}^2(\bs_n,\bullet)\|_{\cH_1^{\otimes n}}^{\frac{1}{H_0}}\Big)^{2H_0}\nonumber\\
\leq &  C_2(n,t-s)  \big(\big|\Delta_{t-s}( x-y, x')\big|+p_{t-s}(x-y) N_{t-s}(x')\big)^2, 
\end{align}
and
\begin{align}\label{gn24}
&\Big(\int_{[s,t]^n}d\bs_n\|g_{s,y+y',t,x+x',n}^2(\bs_n, \bullet)-g_{s,y,t,x+x',n}^2(\bs_n, \bullet)\nonumber\\
&\quad-g_{s,y+y',t,x,n}^2(\bs_n, \bullet)+g_{s,y,t,x,n}^2(\bs_n, \bullet)\|_{\cH^{\otimes n}}^{\frac{1}{H_0}}\Big)^{2H_0}\nonumber\\
\leq &  C_2(n,t-s) \big(\big| R_{t-s}( x-y, x', y')\big|+\big|\Delta_{t-s}(x-y,x')\big|N_{t-s}(y')\nonumber\\
&+\big|\Delta_{t-s}(y-x,y')\big|N_{t-s}(x')+p_{t-s}(x-y)N_{t-s}(x')N_{t-s}(y')\big)^2,
\end{align}
where $\Delta_t$, $R_t$ and $N_t$ are defined as in \eqref{delta} - \eqref{n} and $C_2(n,t-s)$ are the same as in Lemma \ref{lmgn1}.
\end{lemma}
\begin{remark}
In what follows, one may find some structures that are almost the same as in Section \ref{sec32_MD}. However, because of the rough dependence in space, the maximal principle is not valid under Hypothesis \ref{h2}. Therefore, we provide the estimates via a different approach, which involves more careful computations.
\end{remark}
\begin{proof}[Proof of Lemma \ref{lmgn2}]
We divide the proof of this lemma into three steps. In Step 1, we prove inequality \eqref{gn21}, and then inequalities \eqref{gn22} and \eqref{gn23} in Step 2. Finally, the proof of inequality \eqref{gn24} is left in Step 3.

{\bf Step 1.} Fix $s<s_1<\dots <s_n<t$. Taking into account  formulas \eqref{ine_cdt} and \eqref{ine_fcdt}, we can write 
\begin{align}\label{hgk1}
\hg_{s,y,t,x,n}^{2}(\bs_n,& \bxi_n)=p_{t-s}(x-y)\exp\Big(-iy\frac{t-s_{1}}{t-s}\sum_{k_1=1}^n\Big(\xi_{k_1}\prod_{k_2=2}^{k_1}\frac{t-s_{k_2}}{t-s_{k_2-1}}\Big)\Big)\nonumber\\
&\times \exp\bigg[-ix\sum_{j=1}^n\frac{s_{j}-s_{j-1}}{t-s_{j-1}}\sum_{k_1=j}^n\Big(\xi_{k_1}\prod_{k_2=j+1}^{k_1}\frac{t-s_{k_2}}{t-s_{k_2-1}}\Big)\bigg]\nonumber\\
&\times \prod_{j=1}^n\exp\bigg[-\frac{(t-s_j)(s_j-s_{j-1})}{2(t-s_{j-1})}\Big[\sum_{k_1=j}^n\Big(\xi_{k_1}\prod_{k_2=j+1}^{k_1}\frac{t-s_{k_2}}{t-s_{k_2-1}}\Big)\Big]^2\bigg],
\end{align}
where by convention $s_0=s$ and $y_0=y$. Recall that the maximal inequality is not applicable in this situation and we need to apply another method. Notice that the spectral measure $\mu(d\xi)$ of Hilbert space $\cH_1$ has a density $|\xi|^{1-2H_1}$ under Hypothesis \ref{h2}. This allows us to perform a change of variables. For any $j=1,\dots,n$, let 
\[
\eta_j=\sum_{k_1=j}^n\xi_{k_1}\prod_{k_2=j+1}^{k_1}\frac{t-s_{k_2}}{t-s_{k_2-1}}.
\]
Then, it is clear that $\xi_n=\eta_n$ and 
\[
\xi_j=\eta_j-\Big(\frac{t-s_{j+1}}{t-s_{j}}\Big)\eta_{j+1},
\]
for all $j=k+2,\dots,n$. Denote by $\Sigma=\Sigma_n=\frac{\partial \bxi_n}{\partial \pmb{\eta}_n}$ the Jacobian matrix of the transformation $\bxi_n\to \pmb{\eta}_n$. Then, we have
 $\det (\Sigma)=1$ and thus
\begin{align}\label{gk11}
\|g_{s,y,t,x,n}^2(\bs_n, \bullet)\|_{\cH_1^{\otimes{n}}}^2=&c_{H_1}^{n}\int_{\R^{n}}d\bxi_n\prod_{i=1}^n|\xi_i|^{1-2H_1}|\hg_{s,y,t,x,n}^2(\bs_n, \bxi_n)|^2\nonumber\\
=& c_{H_1}^{n}p_{t-s}(x-y)^2\int_{\R^{n}}d\pmb{\eta}_n  \prod_{i=1}^{n-1}\Big|\eta_i-\frac{t-s_{i+1}}{t-s_{i}}\eta_{i+1}\Big|^{1-2H_1}|\eta_n|^{1-2H_1}\nonumber\\
&\times \prod_{i=1}^n\exp\Big(-\frac{(t-s_i)(s_i-s_{i-1})}{(t-s_{i-1})}\eta_i^2\Big).
\end{align}
Using the trivial inequality that $|a+b|^{1-2H_1}\leq |a|^{1-2H_1}+|b|^{1-2H_1}$ for all $H_1\in (0,\frac{1}{2})$ and $a,b\in \R$, we get
\begin{align}\label{gk12}
&\int_{\R^n}d\pmb{\eta}_n  \prod_{i=1}^{n-1}\Big|\eta_i-\frac{t-s_{i+1}}{t-s_{i}}\eta_{i+1}\Big|^{1-2H_1}|\eta_n|^{1-2H_1} \prod_{i=1}^n\exp\Big(-\frac{(t-s_i)(s_i-s_{i-1})}{(t-s_{i-1})}\eta_i^2\Big)\nonumber\\
\leq &\sum_{\bbeta_{n-1}=(\beta_1,\dots,\beta_{n-1})\in \{0,1\}^{n-1}}J_{\bbeta_{n-1}},
\end{align}
where
\begin{align*}
J_{\bbeta_n}:=&\int_{\R^{n}}d\pmb{\eta}_{n} \prod_{i=1}^{n-1}\Big(|\eta_i|^{\beta_i(1-2H_1)}\Big|\frac{t-s_{i+1}}{t-s_{i}}\eta_{i+1}\Big|^{(1-\beta_i)(1-2H_1)} \Big) \times |\eta_n|^{1-2H_1}\nonumber\\
&\times \prod_{i=1}^n\exp\Big(-\frac{(t-s_i)(s_i-s_{i-1})}{(t-s_{i-1})}\eta_i^2\Big)\nonumber\\
=&\int_{\R^{n}}d\pmb{\eta}_n  |\eta_{1}|^{\beta_{1}(1-2H_1)}\prod_{i=2}^{n-1}|\eta_i|^{(1-\beta_{i-1}+\beta_i)(1-2H_1)} |\eta_n|^{(2-\beta_{n-1})(1-2H_1)} \nonumber\\
&\times \prod_{i=1}^{n-1}\Big(\frac{t-s_{i+1}}{t-s_{i}}\Big)^{(1-\beta_i)(1-2H_1)}\prod_{i=1}^n\exp\Big(-\frac{(t-s_i)(s_i-s_{i-1})}{(t-s_{i-1})}\eta_i^2\Big).
\end{align*}
Fix $\bbeta_{n-1}\in \{0,1\}^{n-1}$. Then, we can show that
\begin{align*}
J_{\bbeta_n}\leq &c_1c_2^{n}\Big(\frac{(t-s)}{(t-s_{1})(s_{1}-s)}\Big)^{\frac{1}{2}+\frac{1}{2}\beta_{1}(1-2H_1)}\prod_{i=2}^{n-1}\Big(\frac{(t-s_{i-1})}{(t-s_i)(s_i-s_{i-1})}\Big)^{\frac{1}{2}+\frac{1}{2}(1-\beta_{i-1}+\beta_i)(1-2H_1)}\nonumber\\
&\times\Big(\frac{(t-s_{n-1})}{(t-s_n)(s_n-s_{n-1})}\Big)^{\frac{1}{2}+\frac{1}{2}(2-\beta_{n-1})(1-2H_1)}\prod_{i=2}^{n}\Big(\frac{t-s_{i}}{t-s_{i-1}}\Big)^{(1-\beta_{i-1})(1-2H_1)}.
\end{align*}
After simplification, we get
\begin{align}\label{gk13}
&J_{\bbeta_n}\leq c_1c_2^{n}(t-s)^{\frac{1}{2}+\frac{1}{2}\beta_{1}(1-2H_1)}(t-s_{1})^{\frac{1}{2}(\beta_{2}-1)(1-2H_1)}\prod_{i=2}^{n-1}(t-s_{i})^{\frac{1}{2}(\beta_{i+1} - \beta_{i-1})(1-2H_1)}\\
&\quad\times (t-s_n)^{-\frac{1}{2}+\frac{1}{2}\beta_{n-1}(1-2H_1)}(s_{1}-s)^{-\frac{1}{2} - \frac{1}{2}\beta_{1}(1-2H_1)}\prod_{i=2}^{n}(s_i-s_{i-1})^{-\frac{1}{2} - \frac{1}{2}(1-\beta_{i-1}+\beta_i)(1-2H_1)}.\nonumber
\end{align}
Recalling  that $s<s_{1}<\dots<s_n<t$, we can write
\begin{align}\label{gk14}
&(t-s)^{\frac{1}{2}+\frac{1}{2}\beta_{1}(1-2H_1)}(t-s_{1})^{\frac{1}{2}(\beta_{2}-1)(1-2H_1)}\nonumber\\
&\times \prod_{i=2}^{n-1}(t-s_{i})^{\frac{1}{2}(\beta_{i+1} - \beta_{i-1})(1-2H_1)}(t-s_n)^{-\frac{1}{2}+\frac{1}{2}\beta_{n-1}(1-2H_1)}\nonumber\\
= &(t-s)^{\frac{1}{2}}\prod_{i=1}^{n-1}\Big(\frac{t-s_{i-1}}{t-s_{i+1}}\Big)^{\frac{1}{2}\beta_{i}(1-2H_1)}(t-s_n)^{-\frac{1}{2}}\nonumber\\
\leq &(t-s)^{\frac{1}{2}}\prod_{i=1}^{n-1}\Big(\frac{t-s_{i-1}}{t-s_{i+1}}\Big)^{\frac{1}{2}(1-2H_1)}(t-s_n)^{-\frac{1}{2}}=\Big(\frac{t-s}{t-s_n}\Big)^{1-H_1}.
\end{align}
Therefore, combining \eqref{gk11} - \eqref{gk14}, we have
\begin{align*}
\|g_n^2(\bs_n, \bullet,s,y,&t,x)\|_{\cH_1^{\otimes{n}}}^2\leq c_1c_2^{n}p_{t-s}(x-z)^2(t-s)^{1-H_1}(t-s_n)^{H_1-1}\nonumber\\
&\times \sum_{\bbeta_{n-1}\in \{0,1\}^{n-1}}(s_{1}-s)^{-\frac{1}{2} - \frac{1}{2}\beta_{1}(1-2H_1)}\prod_{i=2}^{n}(s_i-s_{i-1})^{-\frac{1}{2} - \frac{1}{2}(1-\beta_{i-1}+\beta_i)(1-2H_1)}.
\end{align*}
To estimate the time integral of $\|g_n^2(\bs_n, \bullet,s,y,t,x)\|_{\cH_1^{\otimes{n}}}^2$, we cannot use  \cite[Lemma 4.5]{HHNT15} directly. But following the same idea, we also find that
\begin{align}\label{gk1} 
&\int_{[s,t]^n}d\bs_{n}\|g_n^2(\bs_{n}, \bullet,s,y,t,x)\|_{\cH_1^{\otimes{n}}}^{\frac{1}{H_0}}\nonumber\\
\leq &c_1c_2^{n}n!p_{t-s}(x-y)^{\frac{1}{H_0}}(t-s)^{\frac{2H_0+H_1-1}{2H_0}n}\Gamma\Big(\frac{2H_0+H_1-1}{2H_0}n+\frac{2H_0+H_1-1}{2H_0}\Big)^{-1}.
\end{align}
Thus, inequality \eqref{gn21} follows from inequality \eqref{gk1} and Corollary \ref{coro_gmm}.

{\bf Step 2.} Fix $s<s_1<\dots<s_n<t$. Let 
\[
J_n:=\|g_n^{2}(\bs_{n}, \bullet,s,y+y',t,x)-g_n^{2}(\bs_n, \bullet,s,y,t,x)\|_{\cH_1^{\otimes{n}}}^2.
\]
Taking into account  formula \eqref{hgk1}, we obtain the next  equality analogously to \eqref{gk11},
\begin{align}\label{dgnk1}
J_n=&c_{H_1}^{n}\int_{\R^{n}}d\pmb{\eta}_{n}\bigg|p_{t-s}(x-y-y')\exp\Big(-i(y+y')\frac{t-s_{1}}{t-s}\eta_{1}\Big)\nonumber\\
&\quad-p_{t-s}(x-y)\exp\Big(-iy\frac{t-s_{1}}{t-s}\eta_{1}\Big)\bigg|^2\prod_{i=1}^{n-1}|\eta_i-\frac{t-s_{i+1}}{t-s_{i}}\eta_{i+1}|^{1-2H_1}|\eta_n|^{1-2H_1} \nonumber\\
&\times\prod_{i=1}^n\exp\Big(-\frac{(t-s_i)(s_i-s_{i-1})}{(t-s_{i-1})}\eta_i^2\Big)\leq  c_1c_2^{n} (G_1+G_2),
\end{align}
where
\begin{align}\label{g1}
G_1=&|p_{t-s}(x-y-y')-p_{t-s}(x-y)|^{2}\int_{\R^{n}}d\pmb{\eta}_{n}\prod_{i=1}^{n-1}|\eta_i-\frac{t-s_{i+1}}{t-s_{i}}\eta_{i+1}|^{1-2H_1}|\eta_n|^{1-2H_1} \nonumber\\
&\times   \prod_{i=1}^n\exp\Big(-\frac{(t-s_i)(s_i-s_{i-1})}{(t-s_{i-1})}\eta_i^2\Big),
\end{align}
and
\begin{align}\label{g2}
G_2=&p_{t-s}(x-y)^2\int_{\R^{n}}d\pmb{\eta}_{n}\Big|\exp\Big(-i(y+y')\frac{t-s_{1}}{t-s}\eta_{1}\Big)-\exp\Big(-iy\frac{t-s_{1}}{t-s}\eta_{1}\Big)\Big|^2\nonumber\\
&\times  \prod_{i=1}^{n-1}\Big|\eta_i-\frac{t-s_{i+1}}{t-s_{i}}\eta_{i+1}\Big|^{1-2H_1} |\eta_n|^{1-2H_1} \prod_{i=1}^n\exp\Big(-\frac{(t-s_i)(s_i-s_{i-1})}{(t-s_{i-1})}\eta_i^2\Big).
\end{align}
Using inequalities \eqref{gk11} and \eqref{gk1}, and Corollary \ref{coro_gmm}, we can write
\begin{align}\label{ing1}
&\Big(n!\int_{\dT_n^{s,t}}d\bs_n G_1^{\frac{1}{2H_0}}\Big)^{2H_0}\leq C_2(n,t-s)\big(p_{t-s}(x-y-y')-p_{t-s}(x-y)\big)^{2}.
\end{align}
To estimate $G_2$, we apply inequality \eqref{eixz} and get
\begin{align*}
G_2\leq &p_{t-s}(x-y)^2\int_{\R^{n}}d\pmb{\eta}_{n}\Big(|y'|^2\Big|\frac{t-s_{1}}{t-s}\eta_{1}\Big|^2\1_{ \{|\frac{t-s_{1}}{t-s}\eta_{1}y'|\leq	 1\}}+2\times \1_{\{|\frac{t-s_{1}}{t-s}\eta_{1}y'|>1\}}\Big)\nonumber\\
&\times  \prod_{i=1}^{n-1}\Big|\eta_i-\frac{t-s_{i+1}}{t-s_{i}}\eta_{i+1}\Big|^{1-2H_1} |\eta_n|^{1-2H_1} \prod_{i=1}^n\exp\Big(-\frac{(t-s_i)(s_i-s_{i-1})}{(t-s_{i-1})}\eta_i^2\Big)
\end{align*}
Using the same idea as in \eqref{gk12} and \eqref{gk13}, we  deduce that
\begin{align}\label{dg2}
G_2\leq c_1c_2^{n} p_{t-s}(x-y)^2(G_{21}+G_{22}),
\end{align}
where
\begin{align*}
G_{21}=&|y'|^2\sum_{\bbeta_{n-1}\in \{0,1\}^{n-1}}\Big[\int_{\R}d\eta_1|\eta_{1}|^{2+\beta_1(1-2H_1)}\1_{ \{|\frac{t-s_{1}}{t-s}\eta_{1}y'|\leq 1\}}\exp\Big(-\frac{(t-s_1)(s_1-s)}{t-s}\eta_1^2\Big)\Big]\nonumber\\
&\times (t-s)^{-2}(t-s_{1})^{\frac{5}{2}+\frac{1}{2}(\beta_1+\beta_{2}-1)(1-2H_1)}\prod_{i=2}^{n-1}(t-s_{i})^{\frac{1}{2}(\beta_{i+1} - \beta_{i-1})(1-2H_1)}\\
&\times (t-s_n)^{-\frac{1}{2}+\frac{1}{2}\beta_{n-1}(1-2H_1)}\prod_{i=2}^{n}(s_i-s_{i-1})^{-\frac{1}{2} - \frac{1}{2}(1-\beta_{i-1}+\beta_i)(1-2H_1)}
\end{align*}
and
\begin{align*}
G_{22}=&\sum_{\bbeta_{n-1}\in \{0,1\}^{n-1}}\Big[\int_{\R}d\eta_1|\eta_{1}|^{\beta_1(1-2H_1)}\1_{ \{|\frac{t-s_{1}}{t-s}\eta_{1}y'|> 1\}}\exp\Big(-\frac{(t-s_1)(s_1-s)}{t-s}\eta_1^2\Big)\Big]\nonumber\\
&\times (t-s)^{-2}(t-s_{1})^{\frac{5}{2}+\frac{1}{2}(\beta_1+\beta_{2}-1)(1-2H_1)}\prod_{i=2}^{n-1}(t-s_{i})^{\frac{1}{2}(\beta_{i+1} - \beta_{i-1})(1-2H_1)}\\
&\times (t-s_n)^{-\frac{1}{2}+\frac{1}{2}\beta_{n-1}(1-2H_1)}\prod_{i=2}^{n}(s_i-s_{i-1})^{-\frac{1}{2} - \frac{1}{2}(1-\beta_{i-1}+\beta_i)(1-2H_1)}.
\end{align*}
Preforming the change of variable $(\frac{(t-s_1)(s_1-s)}{t-s})^{\frac{1}{2}}\eta_1=\eta$, we can show that
\begin{align*}
&\int_{\R}d\eta_1|\eta_{1}|^{2+\beta_1(1-2H_1)}\1_{ \{|\frac{t-s_{1}}{t-s}\eta_{1}y'|\leq 1\}}\exp\Big(-\frac{(t-s_1)(s_1-s)}{t-s}\eta_1^2\Big)\\
=&2\Big(\frac{t-s}{(t-s_1)(s_1-s)}\Big)^{\frac{3}{2}+\frac{1}{2}\beta_1(1-2H_1)}\int_0^{\big(\frac{(t-s)(s_1-s)}{t-s_1}\big)^{\frac{1}{2}}|y'|^{-1}}d\eta|\eta|^{2+\beta_1(1-2H_1)}e^{-\eta^2}\\
\leq &2\Big(\frac{t-s}{(t-s_1)(s_1-s)}\Big)^{\frac{3}{2}+\frac{1}{2}\beta_1(1-2H_1)}\Big(\int_0^{\infty}d\eta|\eta|^{2+\beta_1(1-2H_1)}e^{-\eta^2}\1_{\{|y'|\leq \sqrt{s-s_1}\}}\\
&+\int_0^{\big(\frac{(t-s)(s_1-s)}{t-s_1}\big)^{\frac{1}{2}}|y'|^{-1}}d\eta|\eta|^{2+\beta_1(1-2H_1)}\1_{\{|y'|>\sqrt{s_1-s}\}}\Big)\\
\leq &c_1\Big[\Big(\frac{t-s}{t-s_1}\Big)^{\frac{3}{2}+\frac{1}{2}\beta_1(1-2H_1)}(s_1-s)^{-\frac{3}{2} - \frac{1}{2}\beta_1(1-2H_1)}\1_{\{|y'|\leq \sqrt{s_1-s}\}}\\
&+\Big(\frac{t-s}{t-s_1}\Big)^{3+\beta_1(1-2H_1)}|y'|^{-3-\beta_1(1-2H_1)}\1_{\{|y'|>\sqrt{s_1-s}\}}\Big].
\end{align*}
It follows that
$
G_{21}\leq c_1c_2^n (G_{21}^1+G_{21}^2),
$
where
\begin{align*}
G_{21}^1= &\sum_{\bbeta_{n-1}\in \{0,1\}^{n-1}}|y'|^2(s_1-s)^{-\frac{3}{2} - \frac{1}{2}\beta_1(1-2H_1)}\1_{\{|y'|\leq \sqrt{s_1-s}\}}(t-s)^{-\frac{1}{2}+\frac{1}{2}\beta_1(1-2H_1)}\\
&\times (t-s_{1})^{1+\frac{1}{2}(\beta_{2}-1)(1-2H_1)}\prod_{i=2}^{n-1}(t-s_{i})^{\frac{1}{2}(\beta_{i+1} - \beta_{i-1})(1-2H_1)} (t-s_n)^{-\frac{1}{2}+\frac{1}{2}\beta_{n-1}(1-2H_1)}\\
&\times\prod_{i=2}^{n}(s_i-s_{i-1})^{-\frac{1}{2} - \frac{1}{2}(1-\beta_{i-1}+\beta_i)(1-2H_1)}
\end{align*}
and 
\begin{align*}
G_{21}^2=&\sum_{\bbeta_{n-1}\in \{0,1\}^{n-1}}|y'|^{-1-\beta_1(1-2H_1)}\1_{\{|y'|>\sqrt{s_1-s}\}}(t-s)^{1+\beta_1(1-2H_1)}\nonumber\\
&\times (t-s_{1})^{-\frac{1}{2}+\frac{1}{2}(\beta_{2} - \beta_1-1)(1-2H_1)}\prod_{i=2}^{n-1}(t-s_{i})^{\frac{1}{2}(\beta_{i+1} - \beta_{i-1})(1-2H_1)}\\
&\times (t-s_n)^{-\frac{1}{2}+\frac{1}{2}\beta_{n-1}(1-2H_1)}\prod_{i=2}^{n}(s_i-s_{i-1})^{-\frac{1}{2} - \frac{1}{2}(1-\beta_{i-1}+\beta_i)(1-2H_1)}.
\end{align*}
Notice that on the set $\{|y'|\leq \sqrt{s_1-s}\}$, 
\[
|y'|^2(s_1-s)^{-1-\frac{1}{2}\beta_1(1-2H_1)}\leq |y'|^{2H_0-\frac{1}{2}}(s_1-s)^{-\frac{1}{4} - \frac{1}{2}\beta_1(1-2H_1)-H_0}.
\]
Combining this fact with inequalities \eqref{gk14} and \eqref{gk1}, and Corollary \ref{coro_gmm}, we get
\begin{align}
\Big(n!\int_{\dT_n^{s,t}}d\bs_n\big|G_{21}^1\big|^\frac{1}{2H_0}&\Big)^{2H_0}\leq c_1c_2^{n}\sum_{\beta_1\in \{0,1\}}|y'|^{2H_0-\frac{1}{2}}\bigg[\int_{[s,t]^n}d\bs_n\Big((s_1-s)^{-\frac{1}{4} - \frac{1}{2}\beta_1(1-2H_1)-H_0}\nonumber\\
&\times \frac{t-s_1}{t-s}\prod_{i=2}^{n}(s_i-s_{i-1})^{-\frac{1}{2} - \frac{1}{2}(1-\beta_{i-1}+\beta_i)(1-2H_1)}\Big)^{\frac{1}{2H_0}}\bigg]^{2H_0}\1_{\{|y'|\leq \sqrt{t-s}\}}\nonumber\\
\leq &C_2(n,t-s)(t-s)^{\frac{1}{4}-H_0}|y'|^{2H_0-\frac{1}{2}}\1_{\{|y'|\leq \sqrt{t-s}\}}.
\end{align}
Following the same arguments, we can also deduce that 
\begin{align}
\Big(n!\int_{\dT_n^{s,t}}d\bs_n&\big|G_{21}^2\big|^\frac{1}{2H_0}\Big)^{2H_0}\leq C_2(n,t-s) \sum_{\beta_1\in\{0,1\}}(t-s)^{\frac{1}{2}-2H_0+\frac{1}{2}\beta_1(1-2H_1)}|y'|^{-1-\beta_1(1-2H_1)}\nonumber\\
&\qquad\times [(t-s)\wedge |y'|^2]^{2H_0}\nonumber\\
\leq &C_2(n,t-s)\big((t-s)^{\frac{1}{4}-H_0}|y'|^{2H_0-\frac{1}{2}}\1_{\{|y'|\leq \sqrt{t-s}\}}+\1_{\{|y'|>\sqrt{t-s}\}}\big)
\end{align}
and
\begin{align}\label{ing22}
\Big(n!\int_{\dT_n^{s,t}}d\bs_n G_{22}^{\frac{1}{2H_0}}\Big)^{2H_0}\leq C_2(n,t-s)\big((t-s)^{\frac{1}{4}-H_0}|y'|^{2H_0-\frac{1}{2}}\1_{\{|y'|\leq \sqrt{t-s}\}}+\1_{\{|y'|>\sqrt{t-s}\}}\big).
\end{align}
Therefore, inequality \eqref{gn22} is a consequence of inequalities \eqref{dgnk1} and \eqref{ing1} - \eqref{ing22}.

Inequality \eqref{gn23} is just another version of \eqref{gn22}, if one make the change of variable $s_i=t+s-u_i$ for all $i=1,\dots, n$. Thus we can conclude that \eqref{gn23} holds  true.

{\bf Step 3.} Due to formulas \eqref{hgk1} and \eqref{gk11}, we can write
\begin{align}\label{d4gn}
&\big\| g_n^2(\bs_n, \bullet, s,y+y',t,x+x')-g_n^2(\bs_n, \bullet,s,y,t,x+x')\nonumber\\
&-g_n^2(\bs_n, \bullet,s,y+y',t,x)+g_n^2(\bs_n, \bullet,s,y,t,x)\big\|_{\cH_1^n}^2\nonumber\\
= &c_{H_1}^n\int_{\R^n}d\pmb{\eta}_n\big|h(x+x',y+y')-h(x+x',y)-h(x,y+y')+h(x,y)\big|^2\nonumber\\
&\times   \prod_{i=1}^{n-1}|\eta_i-\frac{t-s_{i+1}}{t-s_{i}}\eta_{i+1}|^{1-2H_1}|\eta_{n}|^{1-2H_1} \prod_{i=1}^{n}\exp\Big(-\frac{(t-s_i)(s_i-s_{i-1})}{(t-s_{i-1})}\eta_i^2\Big),
\end{align}
where
$
h(x,y)=h_1(x,y)h_2(y)h_3(x),
$
with
\begin{align*}
h_1(x,y)=p_{t-s}(x-y),\quad h_2(y)=\exp\left(-iy\frac{t-s_{1}}{t-s}\eta_{1}\right)
\end{align*}
and
\begin{align*} 
h_3(x)= \exp\left(-ix\sum_{j=1}^{n}\frac{s_{j}-s_{j-1}}{t-s_{j-1}}\eta_j\right).
\end{align*}
Notice that we can bound the rectangular increment as follows:
\begin{align*}
&\big|h(x+x',y+y')-h(x+x',y)-h(x,y+y')+h(x,y)\big|\\
\leq &|h_1(x+x',y+y')-h_1(x+x',y)-h_1(x,y+y')+h_1(x,y)||h_2(y+y')||h_3(x+x')|\\
&+|h_1(x+x',y)-h_1(x,y)||h_2(y+y')-h_2(y)||h_3(x+x')|\\
&+|h_1(x,y+y')-h(x,y)||h_2(y+y')||h_3(x+x')-h_3(x)|\\
&+|h_1(x,y)||h_2(y+y')-h_2(y)||h_3(x+x')-h_3(x)|  :=\hbar_1+\hbar_2+\hbar_3+\hbar_4.
\end{align*}
Following similar arguments as in Step 2, we can estimate the expressions
\begin{align*}
&\bigg[\int_{\dT_n^{s,t}}d\bs_n\bigg(\int_{\R^n}d\pmb{\eta}_{n} \hbar_k^2  \prod_{i=1}^{n-1}|\eta_i-\frac{t-s_{i+1}}{t-s_{i}}\eta_{i+1}|^{1-2H_1}|\eta_n|^{1-2H_1}\nonumber\\
&\times  \prod_{i=1}^{k_2}\exp\Big(-\frac{(t-s_i)(s_i-s_{i-1})}{(t-s_{i-1})}\eta_i^2\Big)\bigg)^{\frac{1}{2H_0}}\bigg]^{2H_0}
\end{align*}
for all $k=1,\dots, 4$ and obtain inequality \eqref{gn24}. The proof of this lemma is completed.
\end{proof}

\subsubsection{Proof of Lemmas \ref{produ} and \ref{prod2u}}\label{ss.pfp1}
Having Lemmas \ref{lmgn1} and \ref{lmgn2}, we are ready to present the proof of Lemmas \ref{produ} and \ref{prod2u}. By \cite[Proposition 1.2.7]{Nualart06}, we can write the chaos expansion for the Malliavin derivatives of $u$ as follows. Fix $(t,x)\in \R_+\times \R$, then for all $s,s_1,s_2\in [s,t]$ and $y,y_1,y_2\in \R$,
\begin{align}\label{cosd1}
D_{s,y}u(t,x)=\sum_{n=0}^{\infty}I_n\big((n+1)f_{t,x,n+1}(\bullet,s,y)\big)
\end{align}
and
\begin{align}\label{cosd2}
D^2_{r,z,s,y}u(t,x)=\sum_{n=0}^{\infty}I_n\big((n+2)(n+1)f_{t,x,n+2}(\bullet,r,s,z,y)\big),
\end{align} 
where $f_{t,x,n}$ is defined as in \eqref{cos2}.
\begin{proof}[Proof of Lemma \ref{produ}]
Suppose that $p=2$. By the chaos expansion \eqref{cosd1} of $D_{s,y}u(t,x)$, in order to prove inequality \eqref{du},  we need to estimate the following expression
\[
\|I_n((n+1)f_{t,x,n+1}(\bullet,s,y))\|_2^2=(n+1)^2n!\|f_{t,x,n+1}(\bullet,s,y)\|_{\kH^{\otimes n}}^2.
\]
Due to the embedding inequality \eqref{emdn}, we know that
\begin{align*}
\|f_{t,x,n+1}(\bullet, s,y)\|_{\kH^{\otimes n}}^2\leq c_{H_0}^n\Big(\int_{[0,t]^n}d\bs_n\|f_{t,x,n+1}(\bs_n,\bullet,s,y)\|_{\cH_1^{\otimes n}}^\frac{1}{H_0}\Big)^{2H_0}.
\end{align*}
Notice that we can decompose the integral region in time as  follows,
\begin{align*}
[0,t]^n=\bigcup_{k=0}^n\{\bs_n\in [0,t]^n, 0<s_{\sigma(1)}<\dots <s_{\sigma(k)}<s<s_{\sigma(k+1)}<\dots s_{\sigma(n)}<t\}\bigcup \cN,
\end{align*}
where $\cN$ is a subset included in $[0,t]^n$ of zero Lebesgue measure. On the other hand, freezing  $0<s_1<s_{k}<s<s_{k+1}\dots<s_n<t$, we have
\begin{align}\label{ftxn+1}
f_{t,x,n+1}(\bs_n,\by_n,s,y)=&\frac{1}{(n+1)!}g_{0, s,y,k}^{1}(\bs_k, \by_k)g_{s,y,t,z,n-k}^2(\bs_{k:n},\by_{k:n})
\end{align}
where $g^1_{k}$ and $g^2_{n-k}$ are defined as in \eqref{gn1} and \eqref{gn2}. If follows that
\begin{align*}
\|f_{t,x,n+1}&(\bullet,s,y)\|_{\kH^{\otimes n}}^2\leq  c_{H_0}^n\Big(\sum_{k=0}^n\binom{n}{k}\int_{[0,r]^n}d\bs_k\int_{[r,t]^{n-k}}d\bs_{k:n}\|f_{t,x,n+1}(\bs_n,\bullet,s,y)\|_{\cH_1^{\otimes n}}^\frac{1}{H_0}\Big)^{2H_0}\\
\leq &\frac{c_1c_2^n}{[(n+1)!]^2}\sum_{k=0}^n{\binom{n}{k}}^{2H_0}\Big(\int_{[0,s]^k}d\bs_k\|g^1_{0,s,y,k}(\bs_{k},\bullet)\|_{\cH_1^{\otimes k}}^\frac{1}{H_0}\Big)^{2H_0}\\
& \times \Big(\int_{[s,t]^{n-k}}d\bs_{k:n}\|g^2_{s,y,t,x,n-k}(\bs_{k:n},\bullet)\|_{\cH_1^{\otimes (n-k)}}^\frac{1}{H_0}\Big)^{2H_0}
\end{align*}
As a consequence of Lemmas \ref{lmgn1} and \ref{lmgn2} and Corollary \ref{coro_gmm}, we deduce that 
\begin{align}\label{ngn}
(n+1)^2n!\|f_{t,x,n+1}(\bullet,s,y)\|_{\kH^{\otimes n}}^2\leq \frac{c_1c_2^nt^{(2H_0+H_1-1)n}p_{t-s}(x-y)^2}{\Gamma(H_1n+1)}.
\end{align}
Finally, it follows from the asymptotic bound of the Mittag-Leffler function (see Lemma \ref{lmm_gmm}-(ii)) that
\begin{align*}
\|D_{r,z}u(t,x)\|_2^2=&\sum_{n=0}^{\infty}\E \big|I_n\big((n+1)f_{t,x,n+1}(\bullet,r,z)\big)\big|^2=\sum_{n=0}^{\infty}(n+1)^2n!\|f_{t,x,n+1}(\bullet,s,y)\|_{\kH^{\otimes n}}^2\\
\leq&  p_{t-r}(x-z)^2\sum_{n=0}^{\infty}\frac{c_1c_2^n t^{(2H_0+H_1-1)n}}{\Gamma(H_1n+1)} \leq c_1\exp\Big(c_2 t^{\frac{2H_0+H_1-1}{2H_1}}\Big)p_{t-r}(x-z)^2.
\end{align*}
This proves inequality \eqref{du} in the case $p=2$. For $p>2$, using \eqref{hyper_ineq}, we can write
\begin{align*}
\|D_{r,z}&u(t,x)\|_p^2\leq \Big(\sum_{n=0}^{\infty}\big\|I_n\big((n+1)f_{t,x,n+1}(\bullet,r,z)\big)\big\|_p\Big)^2\\
\leq &\Big(\sum_{n=0}^{\infty}(p-1)^{\frac{n}{2}}\big\|I_n\big((n+1)f_{t,x,n+1}(\bullet,r,z)\big)\big\|_2\Big)^2\\
\leq & \Big(c_1p_{t-s}(x-y)\sum_{n=0}^{\infty}\frac{c_2^n(p-1)^{\frac{n}{2}}t^{\frac{2H_0+H_1-1}{2}n}}{\Gamma(H_1n+1)^{\frac{1}{2}}}\Big)^2\leq C_1(t)p_{t-s}(x-y)^2.
\end{align*}
The proof of inequality \eqref{du} is completed. 

In the next step, we provide the proof of inequality \eqref{ddu}. Firstly, by chaos expansion \eqref{cosd1} and embedding inequality \eqref{emdn}, we have
\begin{align}\label{djn0}
&\|D_{s,y+y'}u(t,x)-D_{s,y}u(t,x)\|_2^2\nonumber\\
\leq &\sum_{n=0}^{\infty}\frac{c_1c_2^n}{n!}\sum_{k=0}^n {\binom{n}{k}}^{2H_0}\bigg[\Big(\int_{[s,t]^{n-k}}d\bs_{k:n}\|g^2_{s,y,t,x,n-k}(\bs_{k:n},\bullet)\|_{\cH_1^{\otimes (n-k)}}^{\frac{1}{H_0}}\nonumber\\
&\times \int_{[0,s]^k}d\bs_k\|g_{0,s,y+y',k}^1(\bs_{k}, \bullet)-g_{0,s,y,k}^1(\bs_{k}, \bullet)\|_{\cH^{\otimes k}}^{\frac{1}{H_0}}\Big)^{2H_0}\nonumber\\
&+\Big(\int_{[0,s]^{n-k}}d\bs_{k:n}\|g_{s,y+y',t,x,n-k}^2(\bs_{k:n}, \bullet)-g_{s,y,t,x,n-k}^2(\bs_{k:n}, \bullet)\|_{\cH_1^{\otimes (n-k)}}^{\frac{1}{H_0}}\Big)^{2H_0}\nonumber\\
&+\Big(\int_{[0,s]^k}d\bs_k\sup_{z\in \R}\|g^{1}_{0,s,z,k}(\bs_{k},\bullet)\|_{\cH_1^{\otimes k}}^{\frac{1}{H_0}}\Big)^{2H_0}\bigg].
\end{align}
As a consequence of Lemmas \ref{lmgn1}, \ref{lmgn2} and \ref{lmm_gmm} (ii), we obtain inequality \eqref{ddu} for $p=2$ and thus for all $p\geq 2$ due to inequality \eqref{hyper_ineq}. The proof of this lemma is completed.
\end{proof}

\begin{proof}[Proof of Lemma \ref{prod2u}]
It suffices to show this lemma for $p=2$. Denote by $LHS$ the left hand side of \eqref{dd2u}. Then, by the chaos expansion \eqref{cosd2}, we get the following inequality, in the same way as for  \eqref{djn0},
\begin{align}\label{dd2u1}
LHS\leq &\sum_{n=0}^{\infty}\frac{c_1c_2^n}{n!}\sum_{k_2=0}^n\sum_{k_1=0}^{k_2} {\binom{n}{k_2}}^{2H_0}{\binom{k_2}{k_1}}^{2H_0}\\
&\times\Big(\int_{[0,r]^{k_1}}d\bs_{k_1}\int_{[r,s]^{k_2-k_1}}d\bs_{k_1:k_2}\int_{[s,t]^{n-k_2}}d\bs_{k_2:n}(K_1+K_2+K_3+K_4)^{\frac{1}{H_0}}\Big)^{2H_0},\nonumber
\end{align}
where
\begin{align*}
K_1 =&\|g_{s,y,t,x,n-k_2}^2(\bs_{k_2:n},\bullet)\|_{\cH_1^{\otimes (n-k_2)}}\|g_{0,r,z+z',k_1}^{1}(\bs_{k_1},\bullet)-g_{0,r,z,k_1}^1(\bs_{k_1},\bullet)\|_{\cH_1^{\otimes k_1}}\nonumber\\
&\times\| g_{r,z,s,y+y',k_2-k_1}^2(\bs_{k_1:k_2}, \bullet)-g_{r,z,s,y,k_2-k_1}^2(\bs_{k_1:k_2}, \bullet)\|_{\cH_1^{\otimes (k_2-k_1)}},
\end{align*}
\begin{align*}
K_2=&\|g_{s,y,t,x,n-k_2}^2(\bs_{k_2:n},\bullet)\|_{\cH_1^{\otimes (n-k_2)}}\|g_{0,r,z+z',k_1}^{1}(\bs_{k_1},\bullet)\|_{\cH_1^{\otimes k_1}}\nonumber\\
&\times \| g_{r,z+z',s,y+y',k_2-k_1}^2(\bs_{k_1:k_2}, \bullet)-g_{r,z+z',s,y,k_2-k_1}^2(\bs_{k_1:k_2}, \bullet)\nonumber\\
&\quad-g_{r,z,s,y+y',k_2-k_1}^2(\bs_{k_1:k_2}, \bullet)+g_{r,z,s,y,k_2-k_1}^2(\bs_{k_1:k_2}, \bullet)\|_{\cH_1^{\otimes (k_2-k_1)}},
\end{align*}
\begin{align*}
K_3=& \|g_{s,y+y',t,x,n-k_2}^2(\bs_{k_2:n},\bullet)-g_{s,y,t,x,n-k_2}^2(\bs_{k_2:n},\bullet)\|_{\cH_1^{\otimes (n-k_2)}}\\
&\times \| g_{r,z,s,y+y',k_2-k_1}^2(\bs_{k_1:k_2}, \bullet)\|_{\cH_1^{\otimes (k_2-k_1)}}\|g_{0,r,z+z',k_1}^{1}(\bs_{k_1},\bullet)-g_{0,r,z,k_1}^{1}(\bs_{k_1},\bullet)\|_{\cH_1^{\otimes k_1}},\nonumber
\end{align*}
and
\begin{align*}
K_4=& \|g_{s,y+y',t,x,n-k_2}^2(\bs_{k_2:n},\bullet)-g_{s,y,t,x,n-k_2}^2(\bs_{k_2:n},\bullet)\|_{\cH_1^{\otimes (n-k_2)}} \|g_{0,r,z+z',k_1}^{1}(\bs_{k_1},\bullet)\|_{\cH_1^{\otimes k_1}}\nonumber\\
&\times \| g_{r,z+z',s,y+y',k_2-k_1}^2(\bs_{k_1:k_2}, \bullet)- g_{r,z,s,y+y',k_2-k_1}^2(\bs_{k_1:k_2}, \bullet)\|_{\cH_1^{\otimes (k_2-k_1)}}.
\end{align*}
By Lemmas \ref{lmgn1} - \ref{lmgn2} and Corollary \ref{coro_gmm}, we get
\begin{align}
&\Big(\int_{[0,r]^{k_1}}d\bs_{k_1}\int_{[r,s]^{k_2-k_1}}d\bs_{k_1:k_2}\int_{[s,t]^{n-k_2}}d\bs_{k_2:n}K_1^{\frac{1}{H_0}}\Big)^{2H_0}\nonumber\\
\leq & C_2(n,t)p_{t-s}(x-y)^2N_r(z')^2 \big(\big|\Delta_{s-r}(y-z,y')\big|+p_{s-r}(y-z)N_{s-r}(y')\big)^2,
\end{align}
\begin{align}
&\Big(\int_{[0,r]^{k_1}}d\bs_{k_1}\int_{[r,s]^{k_2-k_1}}d\bs_{k_1:k_2}\int_{[s,t]^{n-k_2}}d\bs_{k_2:n}K_2^{\frac{1}{H_0}}\Big)^{2H_0}\nonumber\\
\leq & C_2(n,t)p_{t-s}(x-y)^2 \big(\big|R_{s-r}(y-z,y',z')\big|+\big|\Delta_{s-r}(y-z,y')\big| N_{s-r}(z')\nonumber\\
&+\big|\Delta_{s-r}(z-y,z')\big|N_{s-r}(y')+p_{s-r}(y-z)N_{s-r}(y')N_{s-r}(z')\big)^2,
\end{align}
\begin{align}
&\Big(\int_{[0,r]^{k_1}}d\bs_{k_1}\int_{[r,s]^{k_2-k_1}}d\bs_{k_1:k_2}\int_{[s,t]^{n-k_2}}d\bs_{k_2:n}K_3^{\frac{1}{H_0}}\Big)^{2H_0}\nonumber\\
\leq & C_2(n,t)p_{s-r}(y+y'-z)^2 N_r(z')\big|\big(\big|\Delta_{t-s}(y-x,y')\big|+p_{t-s}(x-y)^2 N_{t-s}( y')\big)^2
\end{align}
and
\begin{align}\label{k4}
&\Big(\int_{[0,r]^{k_1}}d\bs_{k_1}\int_{[r,s]^{k_2-k_1}}d\bs_{k_1:k_2}\int_{[s,t]^{n-k_2}}d\bs_{k_2:n}K_4^{\frac{1}{H_0}}\Big)^{2H_0}\nonumber\\
\leq & C_2(n,t) \big(\big|\Delta_{s-r}( z-y-y',z')\big|+p_{s-r}(y+y'-z) N_{s-r}(z')\big)^2\nonumber\\
&\times\big(\big|\Delta_{t-s}( y-x,y')\big|+p_{t-s}(x-y)N_{t-s}( y')\big)^2.
\end{align}
Therefore, inequality \eqref{dd2u} is a consequence of inequalities \eqref{dd2u1} - \eqref{k4}, Lemma \ref{lmm_gmm} and Corollary \ref{coro_gmm}. This completes the proof of Lemma \ref{prod2u}. 
\end{proof}

\end{document}